\documentclass[11pt]{amsart}

\usepackage{graphicx, amsthm, amsfonts, amscd,epstopdf} 
\usepackage[all]{xy}
\usepackage{hyperref}
\include{header}
 \usepackage{pdfsync} 
 \usepackage{color}

\def\lso{\sigma_0}
\def\sun{\sigma_0}

\newcommand{\rr}{\ensuremath{\mathbb{R}}}

 \theoremstyle{plain}
\newtheorem{thm}{Theorem}[section]
\newtheorem{cor}[thm]{Corollary}
\newtheorem{lem}[thm]{Lemma}

\newtheorem{prop}[thm]{Proposition}

\newtheorem{ques}[thm]{Question}

\theoremstyle{definition}
\newtheorem{defn}[thm]{Definition}

\theoremstyle{remark}
\newtheorem{rem}[thm]{Remark}
\newtheorem{ex}[thm]{Example}

\numberwithin{equation}{section}

\begin{document}

\title[The Surgery Unknotting Number of Legendrian Links ]{The Surgery Unknotting Number of Legendrian Links }

\author[A. B.  Boranda]{A. Bianca Boranda} \address{Bryn Mawr College,
  Bryn Mawr, PA 19010}

  \author[L. Traynor]{Lisa Traynor} \address{Bryn Mawr College,
  Bryn Mawr, PA  19010} \email{ltraynor@brynmawr.edu}

\author[S. Yan]{Shuning Yan} \address{Bryn Mawr College, Bryn
  Mawr, PA 19010}
  
   \thanks{LT is
partially supported by NSF grant DMS-0909021. ABB and 
 SY completed this work  as Bryn Mawr College undergraduates; they received support from NSF grant DMS-0909021
  to pursue this research.}

\begin{abstract}
 The surgery unknotting number
  of a Legendrian link   is defined as the minimal
  number of  particular oriented  surgeries that are required to convert
  the  link into a Legendrian unknot.   
  Lower bounds for the surgery unknotting number  are given in terms of classical
  invariants of the Legendrian link.
  The surgery unknotting number is calculated
  for every Legendrian link  that is topologically a twist knot or a torus link and for every 
  positive, Legendrian rational link. In addition, the surgery unknotting number is calculated for every Legendrian knot in the Legendrian knot atlas
  of Chongchitmate and Ng whose underlying smooth knot has crossing number $7$ or less. In all these calculations,  as long as the Legendrian link 
   of $j$ components  is not topologically
  a slice knot, its surgery unknotting number 
  is equal to the sum of $(j-1)$ and  twice  the smooth $4$-ball genus of the  underlying smooth link. 
  \end{abstract}

\maketitle


\section{Introduction}
\label{sec:intro}

A classical invariant for smooth knots is the 
unknotting number:   
the unknotting number of a diagram of
a knot $K$ is the minimum number of crossing changes required to change the diagram into a 
diagram of the unknot;  the unknotting number of $K$ is the minimum of the unknotting numbers of
all diagrams of $K$.  
In the following, we will define a sugery unknotting number for 
Legendrian knots and links.  

 Legendrian links are smooth links that satisfy
an additional geometric condition imposed by a contact structure.  We will focus on Legendrian links in $\mathbb R^3$ with
its standard contact structure.
The notion
of Legendrian equivalence is more refined than smooth equivalence: there is only one smooth unknot, but there
are an infinite number of Legendrian unknots.  Figure~\ref{fig:3-unknots} shows the front  projections
of three
different Legendrian unknots; the entire  infinite ``tree"  representing all  Legendrian unknots can be found in Figure~\ref{fig:unknots}.

The act of changing a crossing (smoothly passing a knot  through itself) is not a natural operation in a contact manifold.  Instead,
given a Legendrian link, we will attempt to arrive at a Legendrian unknot through a Legendrian ``surgery" operation in which
 two oppositely oriented strands in 
a Legendrian $0$-tangle are replaced by an oriented, Legendrian $\infty$-tangle as illustrated in  Figure~\ref{fig:0-surgery}.  It is shown
in  Proposition~\ref{prop:unknottable} that every Legendrian link can become a Legendrian unknot after a finite number of surgeries.
The surgery unknotting number of a Legendrian link $\Lambda$, $\lso(\Lambda)$, measures the minimal number of these surgeries that are required to convert
$\Lambda$ to a Legendrian unknot; see Definitions~\ref{defn:strings} and \ref{defn:sun}. 
 \begin{figure}
   \centerline{\includegraphics[height=1.3in]{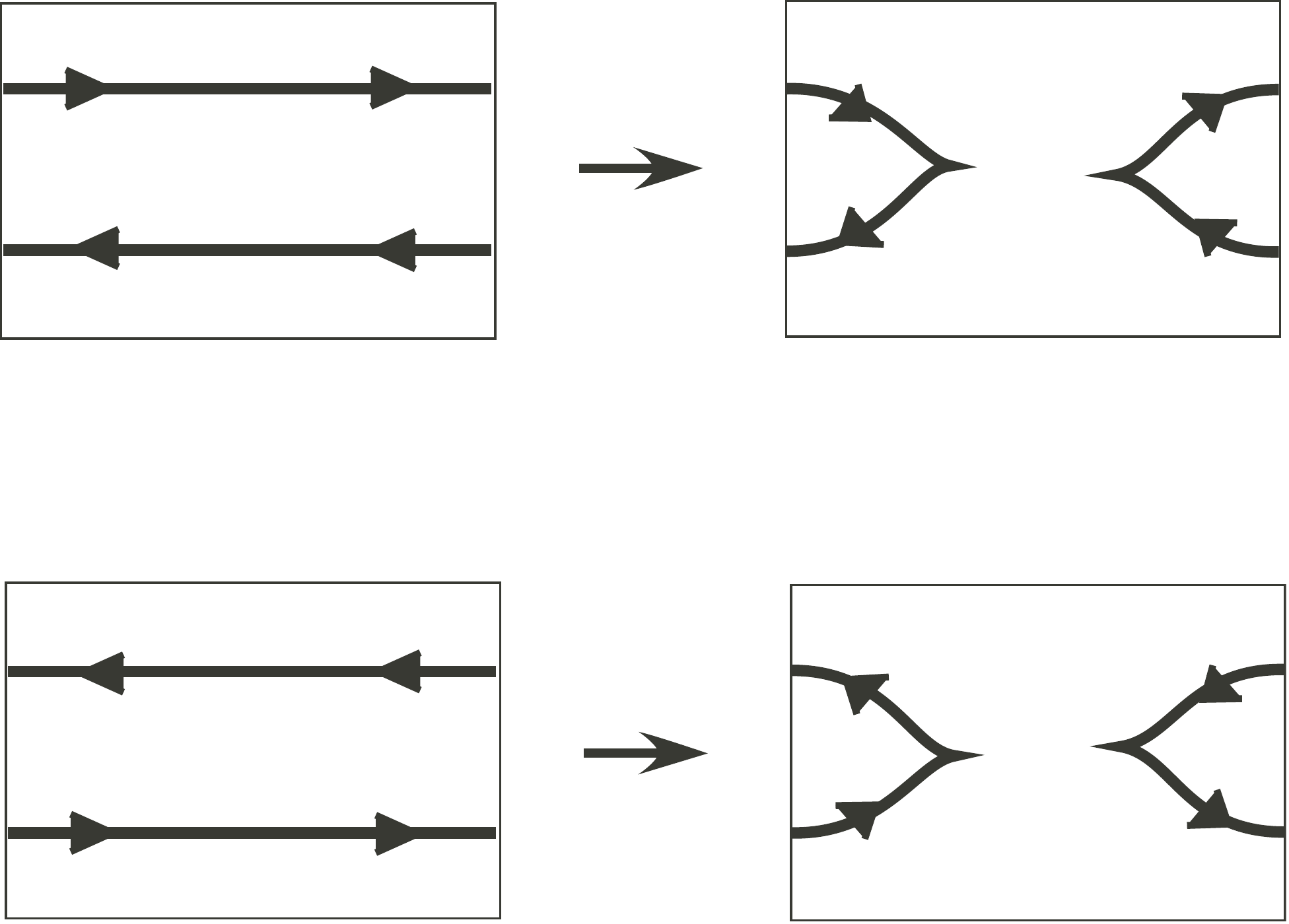}}
  \caption{Oriented Legendrian surgeries:  a basic, compatibly oriented $0$-tangle is replaced by a 
  basic, compatibly oriented $\infty$-tangle.  }
  \label{fig:0-surgery}
\end{figure}
In the following, our goal is to study and calculate this Legendrian invariant $\lso(\Lambda)$.

\subsection{Main Results}
 Lower bounds on $\lso(\Lambda)$ exist in terms of the classical invariants of $\Lambda$.  These invariants include
 invariants of the underlying smooth link type $L_\Lambda$
 and the classical Legendrian invariants of $\Lambda$: the Thurston-Bennequin, $tb(\Lambda)$, and rotation number, $r(\Lambda)$,
 as defined in Section~\ref{sec:background}.

 \begin{thm} \label{thm:lso-lb} Let $\Lambda$ be a Legendrian link.  Then:
  \begin{enumerate}
  \item $tb(\Lambda) + |r(\Lambda)| + 1 \leq \lso(\Lambda);$
  \item  if $\Lambda$ has $j$ components, $L_\Lambda$ 
  denotes the underlying smooth 
    link type of $\Lambda$, and $g_4(L_\Lambda)$ denotes the smooth $4$-ball genus of  $L_\Lambda$
    \footnote{$g_4(L_\Lambda)$ denotes the minimal genus of a smooth, compact, connected, oriented surface $\Sigma 
    \subset B^4$ with
    $\partial \Sigma = L_\Lambda \subset \mathbb R^3 \subset S^3 = \partial B^4$.},   
  then $$2 g_4(L_\Lambda) + (j-1) \leq \lso(\Lambda).$$
  \end{enumerate}
  \end{thm}
  
  \begin{rem}  \label{rem:sl-b} In parallel to Theorem~\ref{thm:lso-lb} (1), when $\Lambda$ is a 
  Legendrian knot with underlying smooth knot type $K_\Lambda$,    
  the well-known Slice-Bennequin Inequality
  says that:
  \begin{equation} \label{ineq:sl-b}
  tb(\Lambda) + |r(\Lambda)| + 1 \leq 2g_4(K_\Lambda).
  \end{equation}
  There are now a number of proofs of this result, but all use deep theory.  
  In  \cite{lisca-matic}, Lisca and Mati\'c prove this using  
  their adjunction inequality obtained by   Seiberg-Witten theory.   See also \cite{akbulut-matveyev}  and \cite{rudolph}. 
  In contrast, the proof of Theorem~\ref{thm:lso-lb} is elementary and is given in
   Lemmas~\ref{lem:tb-lb} and \ref{lem:g4-lb}.
   \end{rem}

When $\Lambda$ is a knot, combining Theorem~\ref{thm:lso-lb}(2) and the Slice-Bennequin Inequality~(\ref{ineq:sl-b}), we find:

 \begin{cor}   \label{cor:realize-g4}
  For any Legendrian knot  
  $\Lambda$,  if  $K_\Lambda$ denotes the smooth knot type of $\Lambda$ then   
  $$tb(\Lambda) + |r(\Lambda)| + 1 \leq 2g_4(K_\Lambda) \leq \lso(\Lambda).$$ 
  Thus when
  $\lso(\Lambda) = tb(\Lambda) + |r(\Lambda)| +  1$, $\lso(\Lambda) = 2g_4(K_\Lambda)$.
  \end{cor}

As we will see below, this corollary sometimes allows us to  calculate the smooth $4$-ball genus of a  knot.  

Using the established lower bounds, we can calculate $\sun(\Lambda)$ when the underlying smooth link type of  $\Lambda$ falls within some
important families.

\begin{thm}  \label{thm:family-sum} 
  \begin{enumerate}
\item   If $\Lambda$ is a Legendrian knot that is topologically a non-trivial twist knot,
  then  $\lso(\Lambda) = 2$. 
 \item   If $\Lambda$ is a $j$-component Legendrian link that is topologically a $(jp,jq)$-torus link,  
 $|p| > q > 1$ and $\gcd(p,q) = 1$, then
  $$\lso(\Lambda) = (|jp|-1)(jq-1).$$
  \end{enumerate}
\end{thm}

 Theorem~\ref{thm:family-sum} is proved in Section~\ref{sec:families} as Theorems~\ref{thm:twist} and \ref{thm:torus}.
 The proof of this theorem relies heavily on the classification of Legendrian twist knots given by Etnyre, Ng, and V\'ertesi, \cite{etnyre-ng-vertesi},
and the classification of Legendrian torus knots by Etnyre and Honda, \cite{etnyre-honda:knots}, which was extended to a classification of
  Legendrian torus links by Dalton, \cite{dalton}.   When $\Lambda$ is topologically a positive torus link, $p > 0$, of maximal Thurston-Bennequin invariant, 
  the calculation of $\lso(\Lambda)$ is obtained realizing the lower bound given in Theorem~\ref{thm:lso-lb} by the Legendrian
  invariants of $\Lambda$.
   Thus by Corollary~\ref{cor:realize-g4}, we are able to 
 to reprove the Milnor conjecture about torus knots, originally proved by Kronheimer and Mrowka:

\begin{cor}[\cite{kronheimer-mrowka}] \label{cor:Milnor}  If $T(p,q)$ is a $(p,q)$-torus knot, $|p| > q > 1$, then 
$$2g_4(T(p,q)) = (|p| - 1)(q-1).$$
 \end{cor}

 By comparing $\sigma_0$ of the Legendrian and $g_4$ of the underlying smooth link type, 
 we can rephrase the conclusions of Theorem~\ref{thm:family-sum} as:
 
\begin{cor}\label{cor:twist-torus}  If $\Lambda$ is a Legendrian link that is topologically  a non-slice twist knot
\footnote{Casson and Gordon proved that  the only twist knots that are slice are the unknot, $6_1$, and $m(6_1)$; \cite{casson-gordon}. }  
or a $j$-component torus link, $L_\Lambda$, 
then $$\lso(\Lambda) = 2g_4(L_\Lambda) + (j-1).$$
\end{cor}

 As an additional family of Legendrian links, we consider ``positive, Legendrian rational links".  These links are defined as
 Legendrian numerator closures of the Legendrian rational tangles
 studied, for example, by the second author, \cite{traynor:strat}, and Schneider, \cite{schneider}.  These links  are positive in the sense that 
  an orientation is chosen on the components so that all the crossings have a 
 positive sign. Such Legendrian links are specified by a vector $(c_n, \dots, c_1)$ of positive
 integers; see Definition~\ref{defn:pos-rat} and Figure~\ref{fig:rat-gen}.
  Lemma~\ref{lem:pos} gives conditions on the $c_i$
 that guarantee that the link is positive.

   \begin{thm}  \label{thm:rat-sum} 
    If $\Lambda(c_n, \dots, c_2, c_1)$  is a positive, Legendrian rational link, then
$$\lso(\Lambda(c_n, \dots, c_2, c_1)) = \sum_{i \text{ odd }} c_{i}  - p(n),$$
where $p(n) $ equals  $1$  when $n$ is odd and equals $0$ when $n$ is even. 
   \end{thm}
   
   This is  proved  in Section~\ref{sec:families}; see Theorem~\ref{thm:pos-rat}.
   
  \begin{rem} When $\Lambda$ is  a positive, Legendrian rational link,  
  the calculation of $\lso(\Lambda)$ is obtained realizing the lower bound given in Theorem~\ref{thm:lso-lb}  given
  by the classical Legendrian invariants of $\Lambda$.  
  Thus by Corollary~\ref{cor:realize-g4}, when 
  $\Lambda(c_n, \dots, c_1)$  is a positive, Legendrian rational knot, 
  Theorem~\ref{thm:rat-sum} gives a formula for twice the smooth $4$-ball genus of the underlying smooth knot.  
    This can be used to get formulas for the smooth $4$-ball genus of a knot in terms of its rational
notation.    In particular,
$$\begin{aligned}
&g_4(5_2) = g_4(N(3,2)) = \frac12 \sigma_0(\Lambda(3,2)) = \frac12 (2) = 1, \\
&g_4(7_5) = g_4(N(3,2,2)) = \frac12 \sigma_0(\Lambda(3,2,2)) = \frac12 (2 + 3 -1) = 2, \\
&g_4(N(5, 244, 4, 16, 3, 104, 2, 12, 1)) =   \frac12(1 + 2 + 3 + 4 + 5-1).
\end{aligned}$$
This is an alternate to formulas for calculating the smooth $4$-ball genus in terms of crossings and Seifert circles
as given by Nakamura in \cite{nakamura}.  In turn, using Nakamura's formula, we see that  when the underlying link type
of $\Lambda(c_n, \dots, c_2, c_1)$ is
a $2$-component link $L_\Lambda$, $$\lso(\Lambda(c_n, \dots, c_2, c_1)) = 2 g_4(L_\Lambda)+ 1;$$
see Remark~\ref{rem:g4-rats}.
\end{rem}

Given the above calculations, it is natural to ask:

  \begin{ques} \label{ques:sun-g4} If $\Lambda$ is a Legendrian knot that is topologically a non-slice knot $K_\Lambda$, is
 $\sigma_0 (\Lambda) = 2 g_4(K_\Lambda)?$ 
 More generally, if $\Lambda$ is a Legendrian link of $j \geq 2$ components that is topologically the link $L_\Lambda$,
  is
 $\sigma_0 (\Lambda) = 2 g_4(L_\Lambda) + (j-1)?$ 
 \end{ques}
 
 To investigate the knot portion of this question, we examined Legendrian representatives of  knots with crossing
 number $7$ or less.     There is not yet
  a Legendrian classification of  all  these knot types, but a conjectured
 classification  is given by Chongchitmate and Ng in \cite{chongchitmate-ng}.

  \begin{prop}\label{prop:small-cross} Assuming the conjectured classification of Legendrian knots in \cite{chongchitmate-ng} \footnote{Potential duplications
 in their atlas will not affect the statement.}, if $\Lambda$ is a Legendrian knot that is topologically a
 non-slice knot $K_\Lambda$ with crossing number $7$ or
 less, 
 $\sigma_0(\Lambda) = 2 g_4 (K_\Lambda).$
 \end{prop}

 The only  non-torus and non-twist
 knots with crossing number at most $7$ are $6_2, m(6_2)$, $6_3 = m(6_3)$, $7_3$, $m(7_3)$, $7_4$, $m(7_4)$, $7_5$, $m(7_5)$,
 $7_6$, $m(7_6)$, $7_7$, and $m(7_7)$. 
 While doing the calculations for Legendrians with these knot types, in general we found that for a Legendrian $\Lambda$ 
 whose underlying smooth knot type $K_\Lambda$ 
 satisfied $g_3(K_\Lambda) = g_4(K_\Lambda)$, where $g_3(K_\Lambda)$ denotes the ($3$-dimensional)
 genus of the knot, 
  it is fairly straight forward to show that 
 $\sigma_0(\Lambda) = 2 g_4 (K_\Lambda)$.   Legendrians that are topologically $7_3, m(7_3), 7_4, m(7_4), 7_5$, and $m(7_5)$ 
 fall into this category.   
 For  the remaining knot types under consideration, the calculation of the smooth $4$-ball genus follows from
 the fact that the topological unknotting number of these knots is equal to $1$.  We show
 that in a front projection of a Legendrian knot, it is possible to locally change 
 any negative crossing to a positive one by $2$ surgeries; see Lemma~\ref{lem:neg-cross-unknot}.  This allowed us to prove 
 Proposition~\ref{prop:small-cross}  in the cases where $\Lambda$ is topologically $6_2, 6_3 = m(6_3), 7_6$, or $7_7$.
 For the remaining cases of $m(6_2)$, $m(7_6)$, and $m(7_7)$, results of \cite{signed-unknot} show that it is not possible to find
 a front projection that can be 
  unknotted at a negative crossing.  However, we found front projections that could
 be unknotted at a positive crossing in a special ``S" or ``hooked-X" form: a positive crossing in
 one of these special forms can be locally changed to a negative crossing by $2$ surgeries; see Lemma~\ref{lem:pos-cross-unknot}.

 \subsection{The Lagrangian Motivation and Discussion}   
 All of our calculations indicate that $\sun(\Lambda)$ is measuring an invariant of the underlying
 smooth link type and that this invariant will be the same for $\Lambda$ and $\Lambda'$ when they represent
 smooth knots that differ by the topological mirror operation.  Below is an explanation for why this may be true.   
 
 Although the definition of the surgery unknotting number has been formulated above combinatorially,
 the motivation comes from trying to understand the flexibility and rigidity of Lagrangian submanifolds of
 a symplectic manifold.  From theory developed by  Bourgeois, Sabloff, and the second author in
 \cite{bourgeois-sabloff-traynor}, 
 the existence
 of an unknotting surgery  string $\left( \Lambda_n, \dots, \Lambda_0\right)$, as defined in Definition~
\ref{defn:strings}, implies the existence of an oriented Lagrangian
 cobordism $\Sigma$ in $\rr \times \rr^3 = \{ (s, x, y, z) \} \cap \{ 0 \leq s \leq n \}$ so that $(\Sigma \cap \{ s = i \} ) = \Lambda_i$, for $i = n, \dots, 0$.  
 Furthermore,  if $\Lambda_0$ is the Legendrian unknot with maximal Thurston-Bennequin invariant,
 this cobordism can be ``filled in" with a Lagrangian $\overline \Sigma \subset \{ s \leq n\}$ so that  $\partial \overline \Sigma = \Lambda_n$.
 In fact, it is shown  by Chantraine in \cite{chantraine} that if $\Lambda_0$ is not the Legendrian unknot with maximal Thurston-Bennequin
 invariant, then the cobordism $\Sigma$ cannot be filled in to $\overline \Sigma$; moreover, when there does exist the filling
 to  $\overline \Sigma$
 and the smooth underlying knot type of $\Lambda_n$ is $K_n$,  then the genus of $\overline \Sigma$  agrees with the smooth $4$-ball genus
 of $K_n$.  
 
 From this Lagrangian perspective, it is a bit more natural to consider  surgery strings $(\Lambda_n, \dots, \Lambda_0)$
 where $\Lambda_0$ is a Legendrian unlink (a trivial link of Legendrian unknots), and define a corresponding  ``surgery unlinking number"; this is a project
 that the second author has begun to pursue with other undergraduates.    
 A Lagrangian analogue of Question~\ref{ques:sun-g4} is:
 
  \begin{ques} \label{ques:lag-g4} If $\Lambda$ is a Legendrian knot with underlying smooth knot type  
   $K_\Lambda$, 
  does there exist a Lagrangian cobordism constructed from Legendrian isotopy and oriented Legendrian surgeries
  between $\Lambda$ and $\Lambda_0$, a Legendrian that is a smooth unlink,  
  that realizes $g_4(K_\Lambda)$?
 \end{ques}

 Any Lagrangian constructed from Legendrian isotopy and oriented Legendrian surgeries would be in ribbon form;  this means that
 the restriction of the height function, given by the $s$ coordinate, to the cobordism  would not have any local maxima in the interior of the cobordism.
So  a positive answer to Question~\ref{ques:lag-g4} would imply  that  the slice genus  
 agrees with the ribbon genus; for some background on this and related  problems, see, for example, \cite{livingston-survey}.

\subsection*{Acknowledgements}  The ideas of this project were inspired by joint work of 
 Josh Sabloff  and the second author.  We are extremely grateful for the many fruitful discussions we have had 
with Sabloff throughout this project.   We also gained much by discussions with 
Chuck Livingston about the smooth $4$-ball genus; we are very thankful for his clear
explanations.  We also thank Paul Melvin and other members of our PACT (Philadelphia Area
Contact/Topology) seminar for useful comments
during a series of presentations on this work.

 \section{Background Information on Legendrian Links} \label{sec:background}
Below is some basic background on Legendrian links.  More information can be found, for example, in
\cite{etnyre:knot-survey}.

The {\bf standard contact
structure} on $\mathbb R^3$ is the field of hyperplanes  $\xi$ where
$\xi_p = \ker(dz - ydx)_p$.  A {\bf Legendrian link} is a submanifold, $L$, of $\mathbb R^3$
diffeomorphic to a disjoint union of circles so that for all $p \in L$,
$T_pL \subset \xi_p$.
It is common to examine Legendrian links from their
$xz$-projections, known as their {\bf front projections}.  A Legendrian link will generically have
an immersed front projection with semi-cubical cusps and no vertical tangents; conversely, any such projection can
be uniquely lifted to a Legendrian link using $y = dz/dx$.  Figure~\ref{fig:trefoils} shows
Legendrian versions of the trefoils $3_1$ and $m(3_1)$.  
\begin{figure}
  \centerline{\includegraphics[height=.8in]{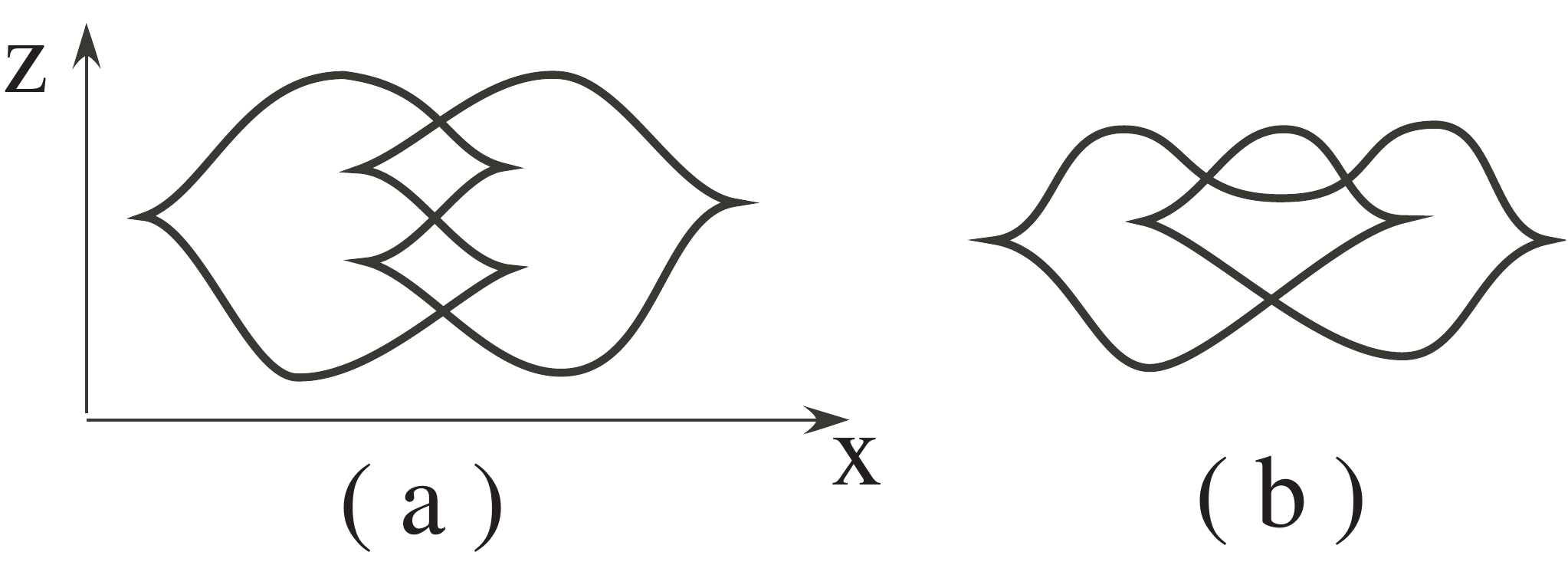}}
  \caption{Front projections of  (a) a Legendrian knot that is topologically the (negative/left) trefoil $3_1$ and
  (b) a Legendrian knot  that is topologically the mirror trefoil $m(3_1)$.  At crossings,
  it is not necessary to specify which strand is the overstrand: the strand with lesser slope will always be on top.}
  \label{fig:trefoils}
\end{figure}

$\Lambda_0$ and $ \Lambda_1$ are {\bf equivalent Legendrian links}
if there exists a $1$-parameter family of Legendrian links $\Lambda_t$ joining $\Lambda_0$ and $\Lambda_1$. 
In fact, Legendrian links $\Lambda_0, \Lambda_1$ are equivalent if and only if their front projections
are equivalent by planar isotopies that do not introduce vertical tangents  and 
the {\bf Legendrian Reidemeister moves} as  shown in Figure~\ref{fig:L-R-moves}.
\begin{figure}
  \centerline{\includegraphics[height=1.2in]{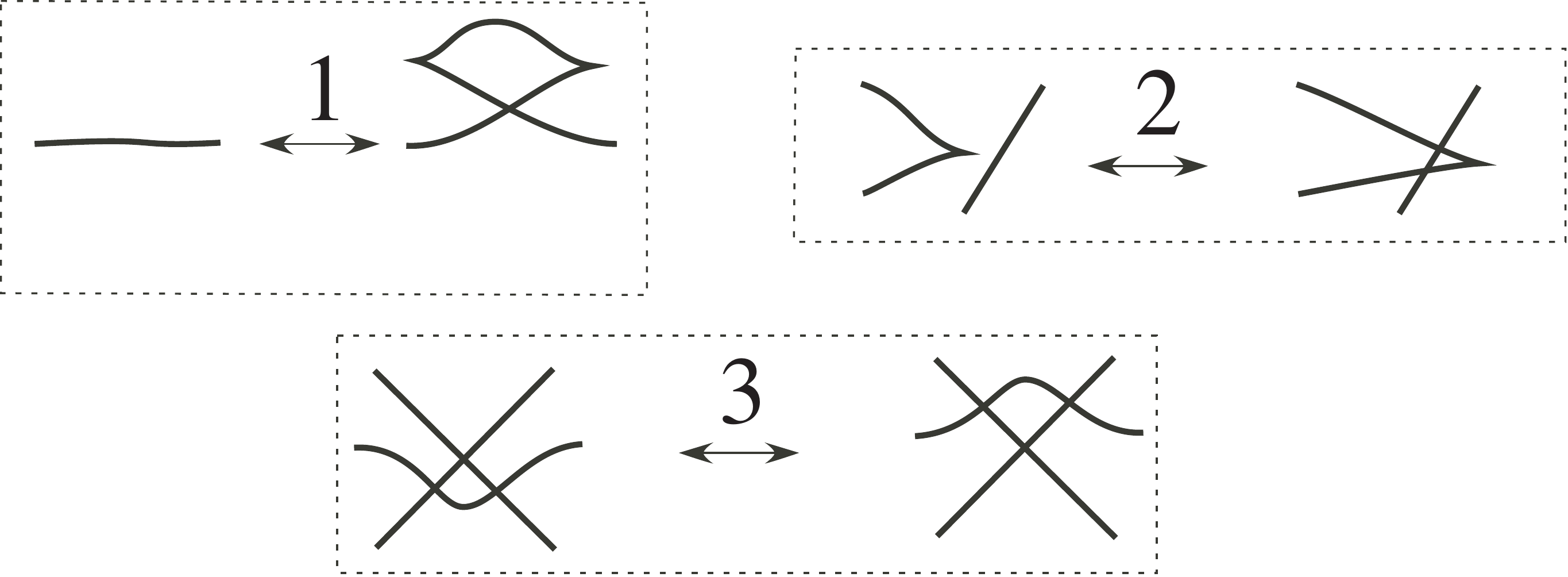}}
  \caption{The three Legendrian Reidemeister moves:  there is another type $1$ move obtained by flipping the 
 planar figure about a horizontal line, and there are three additional type $2$ moves obtained by flipping the 
  planar figure about a vertical, a horizontal, and both a vertical and horizontal line.}
  \label{fig:L-R-moves}
\end{figure}

Every Legendrian knot and link has a Legendrian representative.  In fact, every Legendrian knot
and link has an infinite number of different Legendrian representatives.  For example, Figure~\ref{fig:3-unknots}
shows three different Legendrians that are all topologically the unknot.  These unknots can be distinguished
by classical Legendrian invariants, the Thurston-Bennequin and rotation number.  
These invariants can easily be computed from a front projection of the Legendrian link once we understand how to
 assign a $\pm$ sign to each crossing and an up/down direction to each cusp.
  \begin{figure}
  \centerline{\includegraphics[height=1.2in]{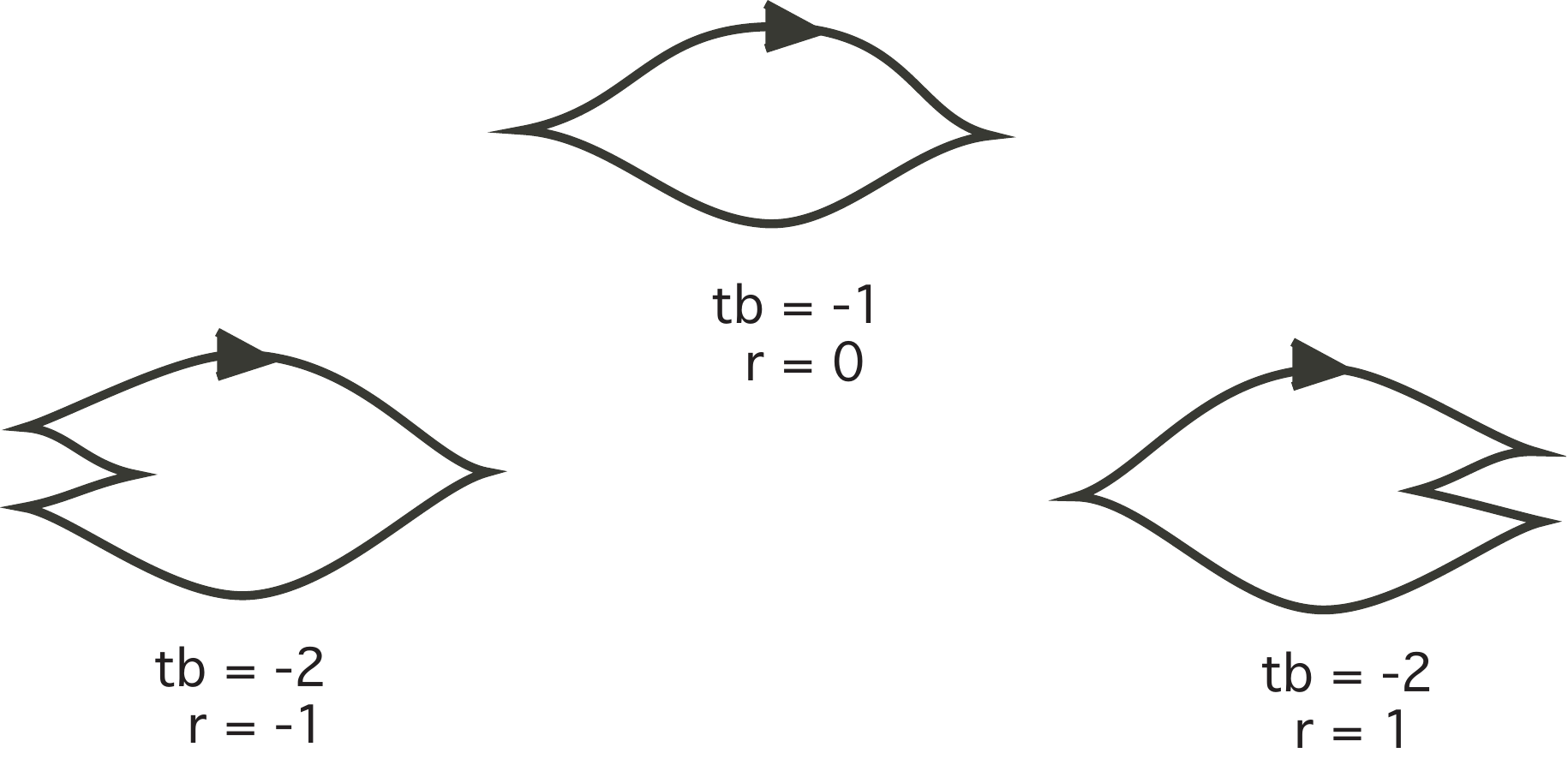}}
  \caption{Three different Legendrian knots that are all topologically the unknot.  } 
  \label{fig:3-unknots}
\end{figure}

 A {\bf positive (negative) crossing} of a front projection of an oriented Legendrian link is a crossing where  the strands  point to the  same side (opposite sides) of a
 vertical line passing through the crossing point; see figure~\ref{fig:positive_negative_crossings}.
Each cusp can also be assigned an {\bf up} or {\bf down} direction; see Figure~\ref{fig:up-down-cusps}.
Then for an oriented Legendrian link $\Lambda$, 
 we have the following formulas for the {\bf Thurston-Bennequin}, $tb(\Lambda)$,  and {\bf rotation number}, $r(\Lambda)$,  invariants:
 \begin{equation} \label{eqn:tb-r}
 tb(\Lambda) = P - N - R, \qquad r(\Lambda) = \frac12(D - U),
 \end{equation}
 where  
 $P$ is the number of positive crossings,     
 $N$ is the number of negative crossings,   
 $R$ is the number of right cusps,  
 $D$ is the number of down cusps,  
 and $U$ is the number of up cusps in a front projection of $\Lambda$.   
 Given that two front projections of equivalent Legendrian links
 differ by the Legendrian Reidemeister moves described in Figure~\ref{fig:L-R-moves}, it is easy to verify that 
 $tb(\Lambda)$ and $r(\Lambda)$ are Legendrian
link invariants.

  \begin{figure}
  \centerline{\includegraphics[height=.8in]{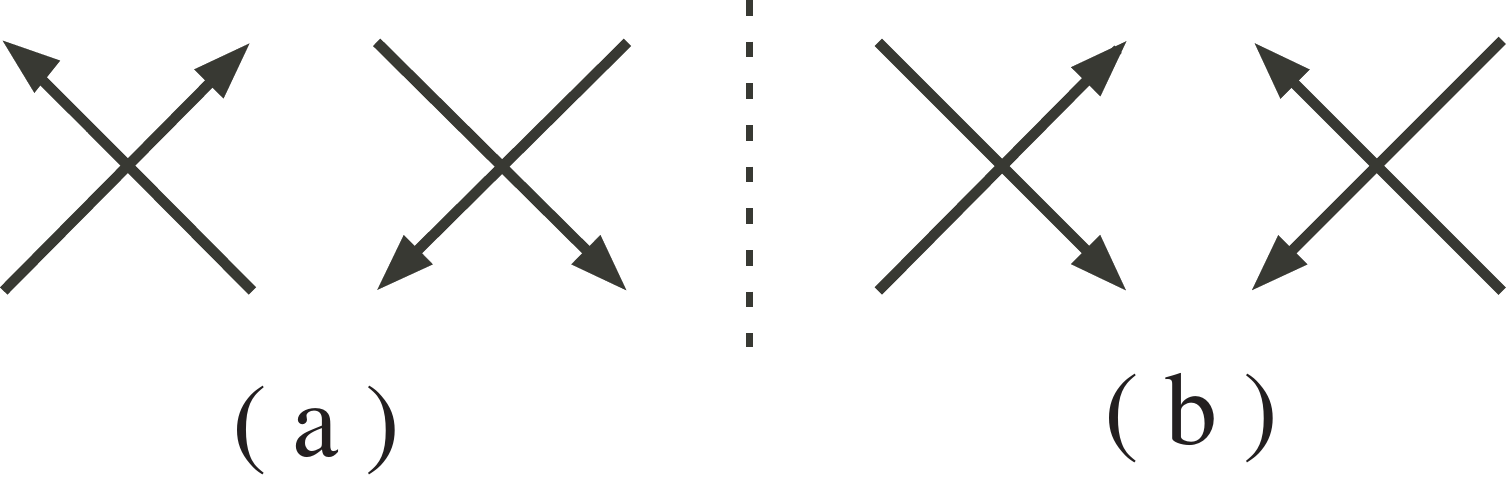}}
  \caption{(a) Negative crossings; (b) Positive crossings. }
  \label{fig:positive_negative_crossings}
\end{figure}

   \begin{figure}
  \centerline{\includegraphics[height=.8in]{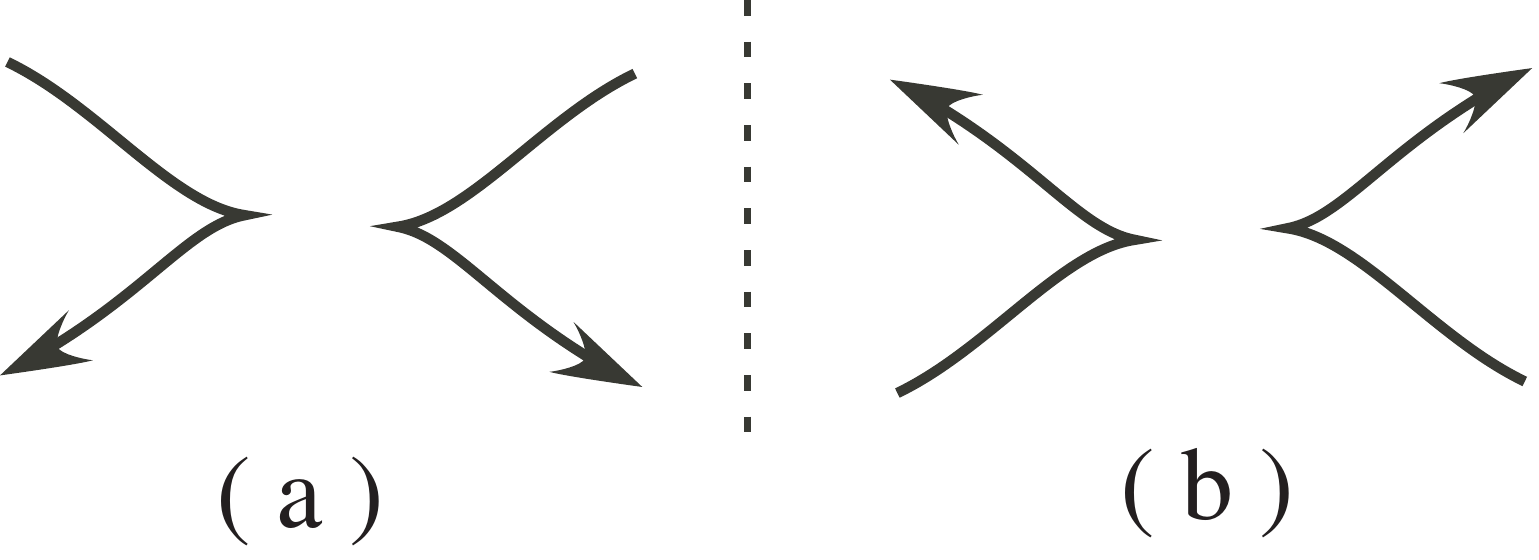}}
  \caption{(a) Right and left down  cusps;  (b) Right and left up  cusps. }
  \label{fig:up-down-cusps}
\end{figure}

 The two unknots in  the second line of Figure~\ref{fig:3-unknots} are obtained from the one at the top
 by adding  an up or down {\bf zig-zag}  (also known as a {\bf $\mp$ stabilization}). 
   In general, this stabilization procedure will not  change the underlying smooth knot type but will
 will decrease the Thurston-Bennequin number by $1$; adding an up (down) zig-zag will decrease (increase)
 the rotation number by $1$.
 If $\Lambda$ is a Legendrian knot, 
 we will use the notation $S_{\pm}(\Lambda)$ to denote the {\bf double stabilization} of $\Lambda$, the Legendrian  knot obtained by adding
 both a positive and negative zig-zag.

 In fact, as discovered by Eliashberg and Fraser, all Legendrian unknots are classified by their Thurston-Bennequin and
 rotation numbers:

 \begin{thm}[\cite{eliashberg-fraser}, \cite{etnyre-honda:knots}]  \label{thm:unknot-classification} 
 Suppose $\Lambda_0$ and $\Lambda_0'$ are oriented Legendrian knots that are both topologically
 the unknot.  Then $\Lambda_0$ is equivalent to $\Lambda_0'$ if and only if $tb(\Lambda_0) = tb(\Lambda_0')$ and
 $r(\Lambda_0) = r(\Lambda_0')$.
 \end{thm}

 \begin{figure}
  \centerline{\includegraphics[height=1in]{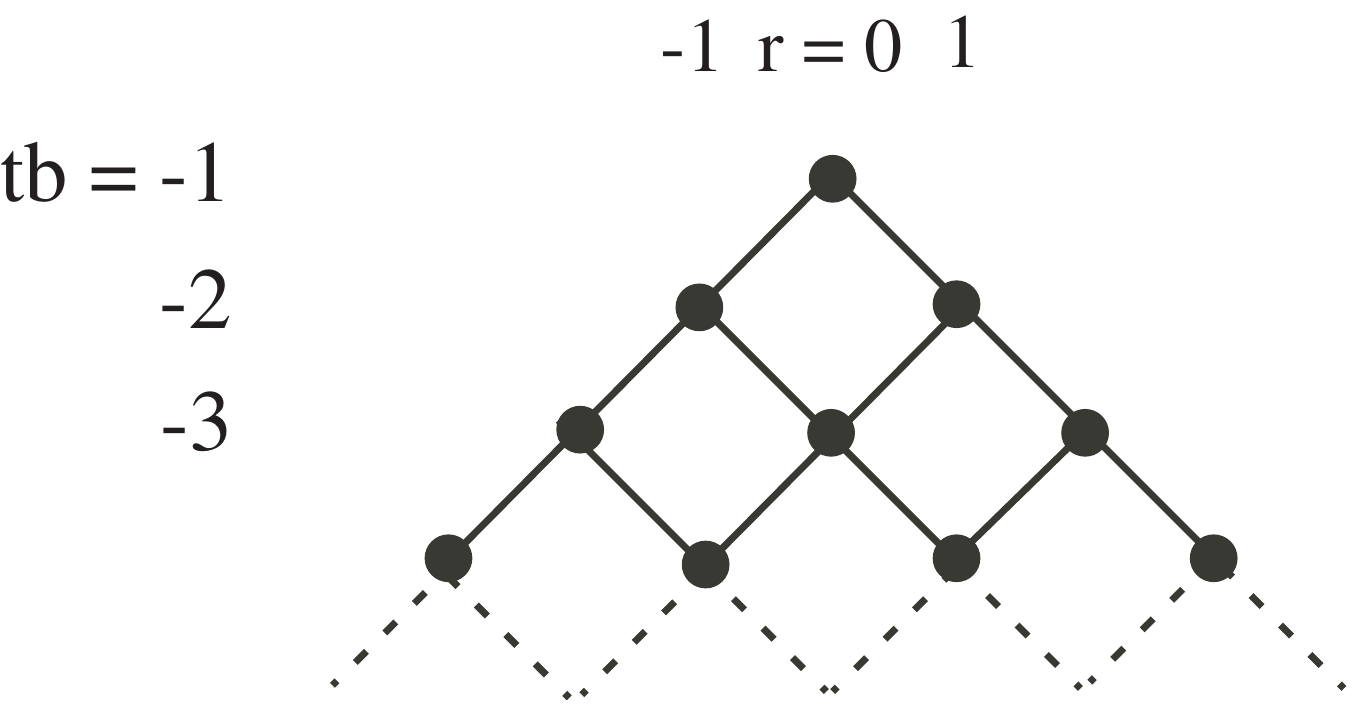}}
  \caption{The tree of all Legendrian unknots. }
  \label{fig:unknots}
\end{figure}

 Figure~\ref{fig:unknots} describes all the Legendrian unknots.  Notice that any Legendrian unknot is equivalent to one that is
 obtained by adding up and/or down zig-zags to the unknot with Thurston-Bennequin number equal to $-1$ and rotation number equal to 
 $0$ shown in Figure~\ref{fig:3-unknots}. 
 
 In general, it is an important question to understand the ``geography" of other knot types.  By work of Etnyre and Honda, \cite{etnyre-honda:knots} and
Etnyre, Ng, and V\'ertesi,  \cite{etnyre-ng-vertesi}, we understand the trees/mountain ranges for all torus and twist knots.  The Legendrian knot atlas
 of Chongchitmate and Ng, \cite{chongchitmate-ng}, gives the known and conjectured mountain ranges for all Legendrian knots with arc index at most $9$; this
 includes all knot types with crossing number at most $7$ and all non-alternating knots with crossing number at most $9$.

 \section{The Surgery Unknotting Number}
 
 In this section, we define the surgery operation, show that every Legendrian link can be unknotted by surgeries, define the
 surgery unknotting number, and give some basic properties of the surgery unknotting number.

 The surgery operation can be viewed as a ``tangle surgery": the replacement
 of one Legendrian tangle by another.
  A {\bf basic, compatibly-oriented Legendrian $0$-tangle} is a Legendrian tangle that is topologically the 0-tangle where the strands
 are oppositely oriented and each strand has  neither crossings nor cusps; the two basic, compatibly-oreinted Legendrian $0$-tangles
 can be seen on the left side of Figure~\ref{fig:0-surgery}.
 A   {\bf basic, compatibly-oriented Legendrian $\infty$-tangle}
 is  Legendrian tangle that is topologically the $\infty$-tangle where the strands are oppositely oriented and each strand has precisely
 $1$ cusp and no crossings;   the two basic, compatibly-oriented Legendrian $\infty$-tangles
 can be seen on the right side of Figure~\ref{fig:0-surgery}.

 \begin{defn} \label{defn:strings}
 An {\bf oriented, Legendrian surgery} of an oriented, Legendrian link is the Legendrian link obtained by replacing a 
 basic, compatibly-oriented Legendrian
 $0$-tangle  with a basic, compatibly-oriented Legendrian $\infty$-tangle;
 see Figure~\ref{fig:0-surgery}.   An {\bf oriented surgery string} consists of a vector of oriented, Legendrian links
 $\left( \Lambda_n, \Lambda_{n-1}, \dots, \Lambda_0\right)$ where, for all $j \in \{n-1, \dots, 0\}$, $\Lambda_{j}$ is obtained from
 $\Lambda_{j+1}$ by an oriented, Legendrian surgery.  An {\bf oriented, unknotting surgery string of length $n$
 for $\Lambda$} consists of an oriented surgery string
 $\left(\Lambda_n, \Lambda_{n-1}, \dots, \Lambda_0\right)$ where $\Lambda_n = \Lambda$ and $\Lambda_0$ is topologically an unknot.
 \end{defn}

 To start, we have the following relationships between the classic invariants of two Legendrian links related by surgery:
   
    \begin{lem} \label{lem:parity+tb}  If $\Lambda$ is an oriented, Legendrian link and $\Lambda'$ is obtained from $\Lambda$ by an oriented, Legendrian surgery, then:
    \begin{enumerate}
    \item the parity of the number of components of $\Lambda$ and $\Lambda'$ differ; 
    \item $tb(\Lambda') = tb(\Lambda) - 1$, and $r(\Lambda') = r(\Lambda)$.
 \end{enumerate}
 \end{lem}
 
 \begin{proof} The statements about the Thurston-Bennequin and rotation numbers are easily verified using
 Equation~(\ref{eqn:tb-r}).  
Regarding the parity,  
  one surgery to a knot will always produce a link of two components, while doing a surgery to a link will increase or decrease the number of components
  by $1$ depending on whether or not the strands in the $0$-tangle belong to the same component of the link.
  \end{proof}

Recall that for any Legendrian knot $\Lambda$, the Legendrian knot $\Lambda' = S_\pm(\Lambda)$ obtained as the
double $\pm$ stabilization of $\Lambda$  will have $r(\Lambda') = r(\Lambda)$ and
$tb(\Lambda') = tb(\Lambda) -2$.  Thus it is potentially possible that $\Lambda'$ can be obtained from $\Lambda$ by two oriented Legendrian
surgeries.  In fact, it is possible.

\begin{lem} \label{lem:2-steps-down}  For any oriented, Legendrian knot $\Lambda$  
there exists an oriented surgery string 
$(\Lambda_2, \Lambda_1, \Lambda_0)$ with $\Lambda_2 = \Lambda$ and $\Lambda_0 = S_\pm(\Lambda)$.
\end{lem}

\begin{proof} These surgeries are illustrated in Figure~\ref{fig:2-steps-down}.
Every Legendrian link $\Lambda$ must have a right cusp. By a Legendrian isotopy, we call pull a right cusp far to
the right and perform one surgery near this right cusp.  This produces a link consisting of the original
link and a Legendrian unknot.  After a Legendrian isotopy, a second surgery can be done using one
strand near the same cusp of the original link and a strand from the unknot.  The result is $S_\pm(\Lambda)$.
\end{proof}

 \begin{figure}[htbp]\label{fig:2-steps-down}
\begin{center}
 \centerline{\includegraphics[height=1.5in]{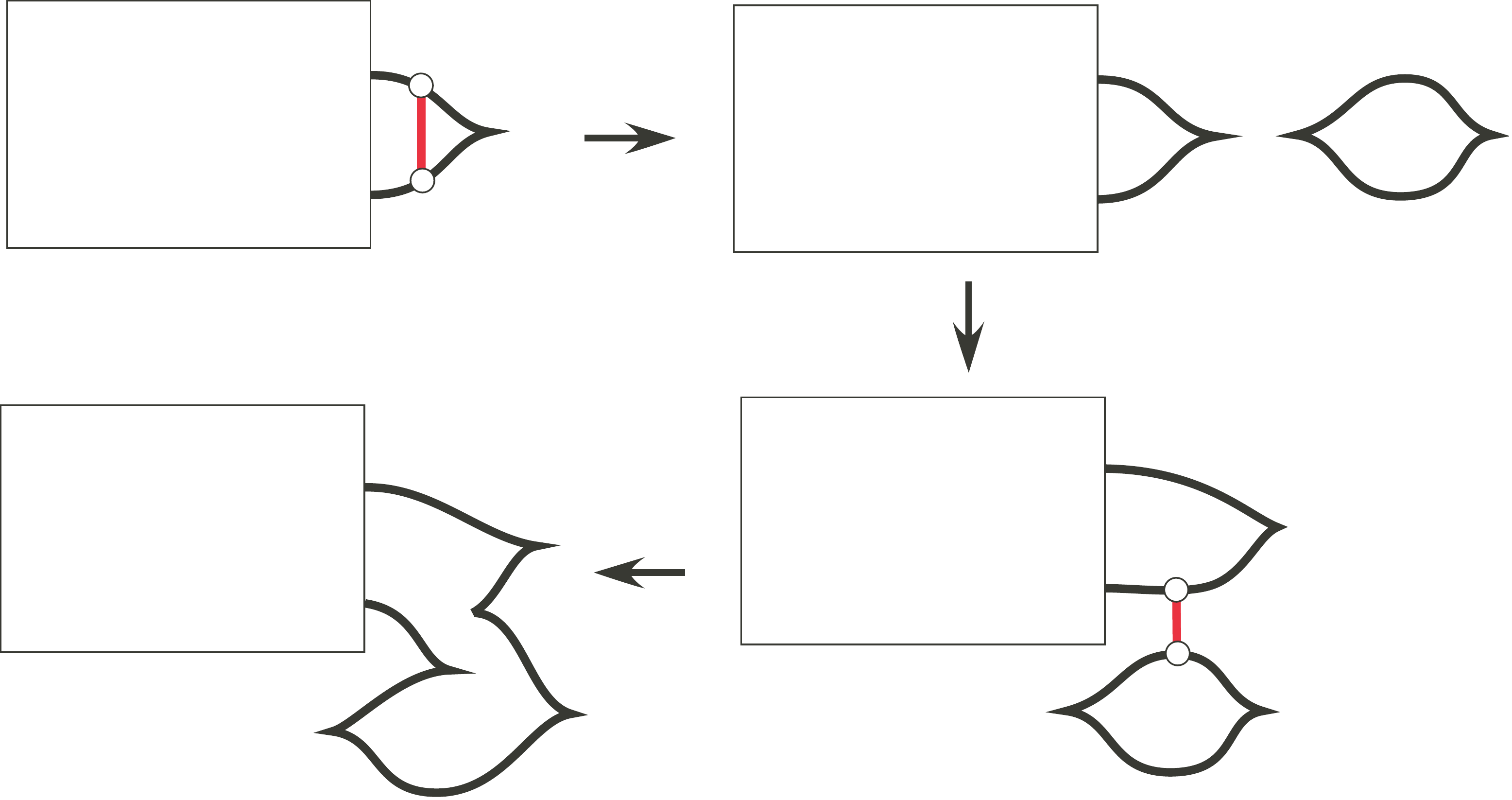}}
\caption{Two oriented, Legendrian surgeries produce $S_\pm(\Lambda)$ from $\Lambda$.}
\end{center}
\end{figure}

In the chart of Legendrian unknots given in Figure~\ref{fig:unknots}, we see that any two unknots
with the same rotation number are related by a sequence of double $\pm$ stabilizations.  Thus we get:

 \begin{cor}   If $\Lambda$ and $\Lambda'$ are oriented, Legendrian unknots with $r(\Lambda) = r(\Lambda')$ and
$tb(\Lambda) = tb(\Lambda') + 2m$, then there exists an oriented surgery string \newline
$\left(\Lambda_{2m}, \Lambda_{2m-1}, \dots,  \Lambda_{0}\right)$, 
where  $\Lambda_{2m} = \Lambda$, and $\Lambda_{0} = \Lambda'$. 
\end{cor}
  
  Thus if we can reach a Legendrian unknot by surgeries, then we can reach an infinite number of Legendrian
  unknots by surgery.   The basis for our new invariant is the fact that every Legendrian link can be ``unknotted" by a string of
 surgeries:
 
 \begin{prop} \label{prop:unknottable}
 For any oriented, Legendrian link $\Lambda$,  there exists an oriented, unknotting surgery string
 $\left( \Lambda = \Lambda_{u}, \Lambda_{u-1}, \dots, \Lambda_0\right)$. 
 Moreover, if $\Lambda$ has $j$ components
 and there exists a front projection of $\Lambda$ with $m$ crossings, then $u \leq 2m + j - 1$.
 \end{prop}

 \begin{proof}  Assume that there is a front  projection of $\Lambda$ with $m$ crossings.  We will first show that 
 there is an oriented surgery string $\left(\widetilde \Lambda_m, \widetilde \Lambda_{m-1}, \dots, \widetilde \Lambda_0\right)$,
 where $\widetilde \Lambda_m = \Lambda$ and 
 $\widetilde \Lambda_0$ is a trivial link of Legendrian unknots.     If $\widetilde \Lambda_0$ has $c$ components, we will then show that
 it is possible to do an additional $c-1$  surgeries to get this into a single unknot.  
   
  Given the initial Legendrian link $\Lambda$ having a projection with $m$ crossings, assume that $n$ of these crossings
  are negative.  It is then possible to construct a surgery string $\left(\widetilde \Lambda_m, \widetilde \Lambda_{m-1}, \dots, \widetilde \Lambda_{m-n}\right)$ where
  $\widetilde \Lambda_m = \Lambda$  and $\widetilde \Lambda_{m-n}$ has a front projection with $m - n$ crossings, all of which are positive.
  This surgery string is obtained by doing a surgery
  to the right of each negative crossing as shown in Figure~\ref{fig:neg-cross}, and then doing a Legendrian isotopy to
  remove the positive crossing introduced by the surgery.   
  Next, by applying a 
  planar Legendrian isotopy, it is possible to assume that all the crossings of $\widetilde \Lambda_{m-n}$ have distinct $x$-coordinates.  
  The left cusps associated to the  leftmost positive crossing are either nested or stacked and fall into one of the $6$ cases
 listed in Figure~\ref{fig:positive-cross}.     
  For each case, it is possible to do a surgery immediately to the right of this leftmost crossing. After Legendrian Reidemeister moves, the crossing is eliminated
  and the number of crossings of the projection of the resulting link has decreased by 1; see Figure~\ref{fig:positive-cross}.  What was
  the second leftmost  positive crossing is now the leftmost positive crossing and the procedure can be repeated.  
  In this way, we obtain a surgery string of Legendrian links 
  $\left(\widetilde \Lambda_m, \dots, \widetilde \Lambda_{m-n}, \widetilde \Lambda_{m-n-1}, \dots, \widetilde \Lambda_0\right)$
  where $\widetilde \Lambda_0$ has a front projection with no crossings.  It follows that $\widetilde \Lambda_0$ is topologically a trivial link of unknots.  
    By applying a Legendrian isotopy, we can assume that $\widetilde \Lambda_0$ consists of $c$ Legendrian unknots which are
 vertically stacked and where each unknot
 is oriented ``clockwise"; an example of this is shown in Figure~\ref{fig:unknot-stack}. It is then easy to see
 that after applying $c-1$ additional surgeries, we can  obtain a Legendrian unknot.  Thus there is a length $u = m + c-1$ 
 unknotting surgery sequence for $\Lambda$. 
   By Lemma~\ref{lem:parity+tb}, if $\Lambda = \widetilde \Lambda_m$ has
 $j$ components, $\widetilde \Lambda_0$ has at most $c= j+m$ components.  Thus we see that $u \leq 2m +j - 1$, as claimed. \end{proof}

   \begin{figure}
 \centerline{\includegraphics[height=0.5in]{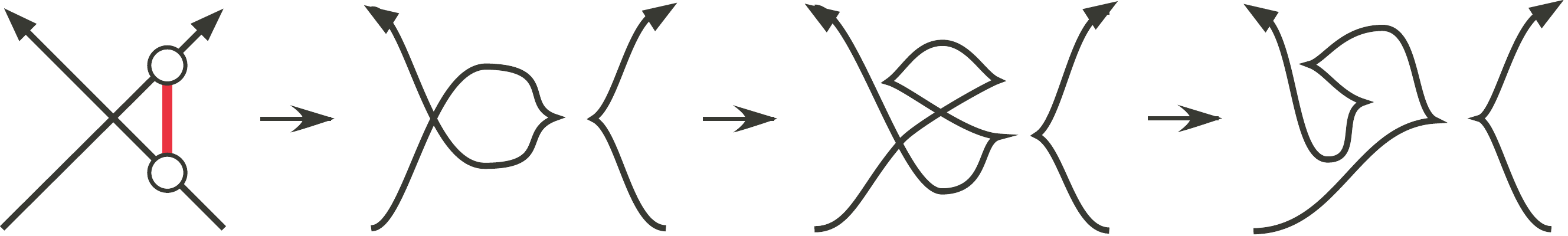}}
  \caption{A negative crossing can be removed by an oriented Legendrian surgery and then Legendrian isotopy.}
  \label{fig:neg-cross}
\end{figure}

   \begin{figure}
  \centerline{\includegraphics[height=2in]{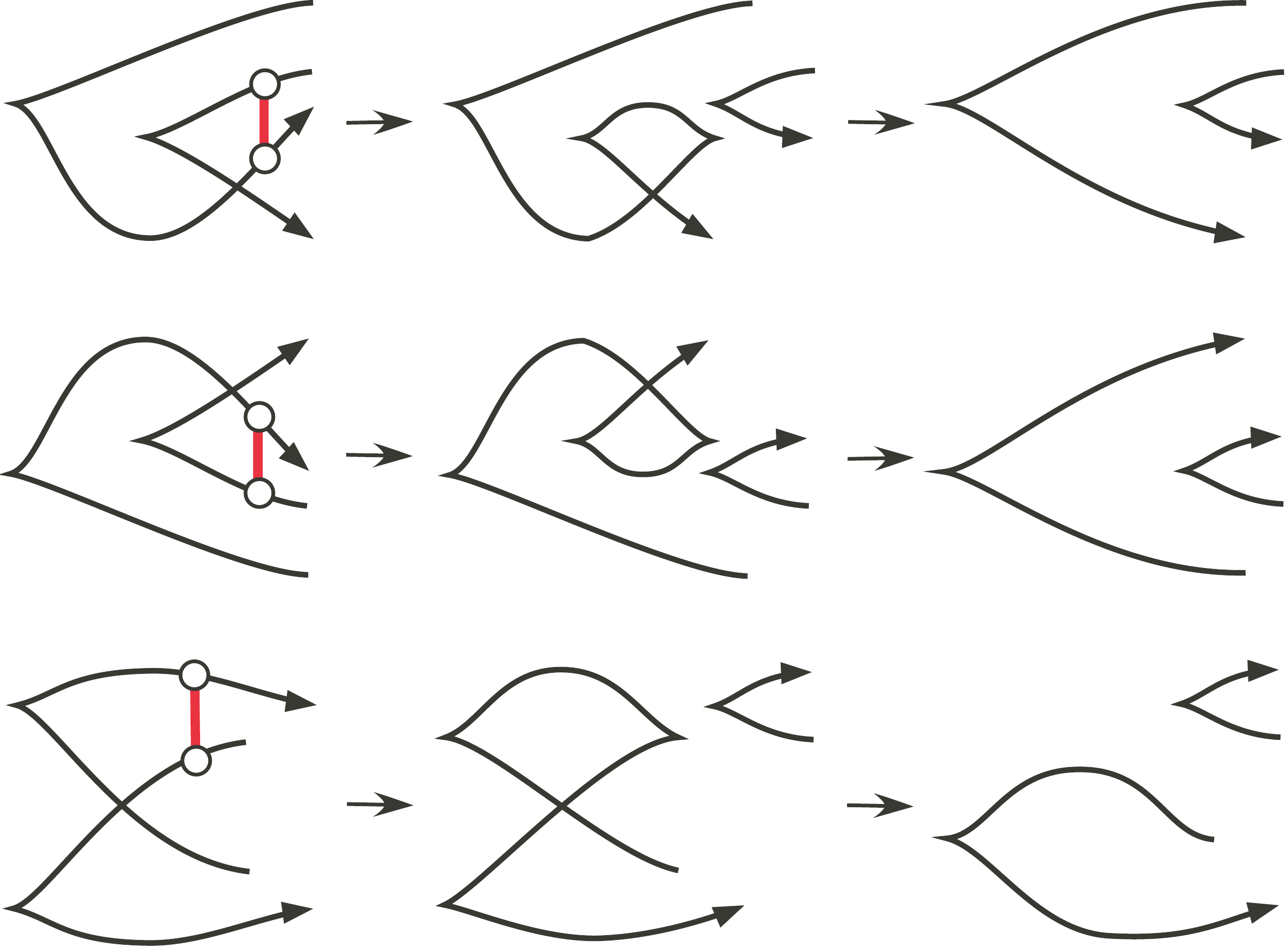}}
  \caption{Three cases for the leftmost positive crossing and their associated left cusps; three additional cases
  are obtained by reversing the orientations on both strands.}
  \label{fig:positive-cross}
\end{figure}

   \begin{figure}
 \centerline{\includegraphics[height=1.2in]{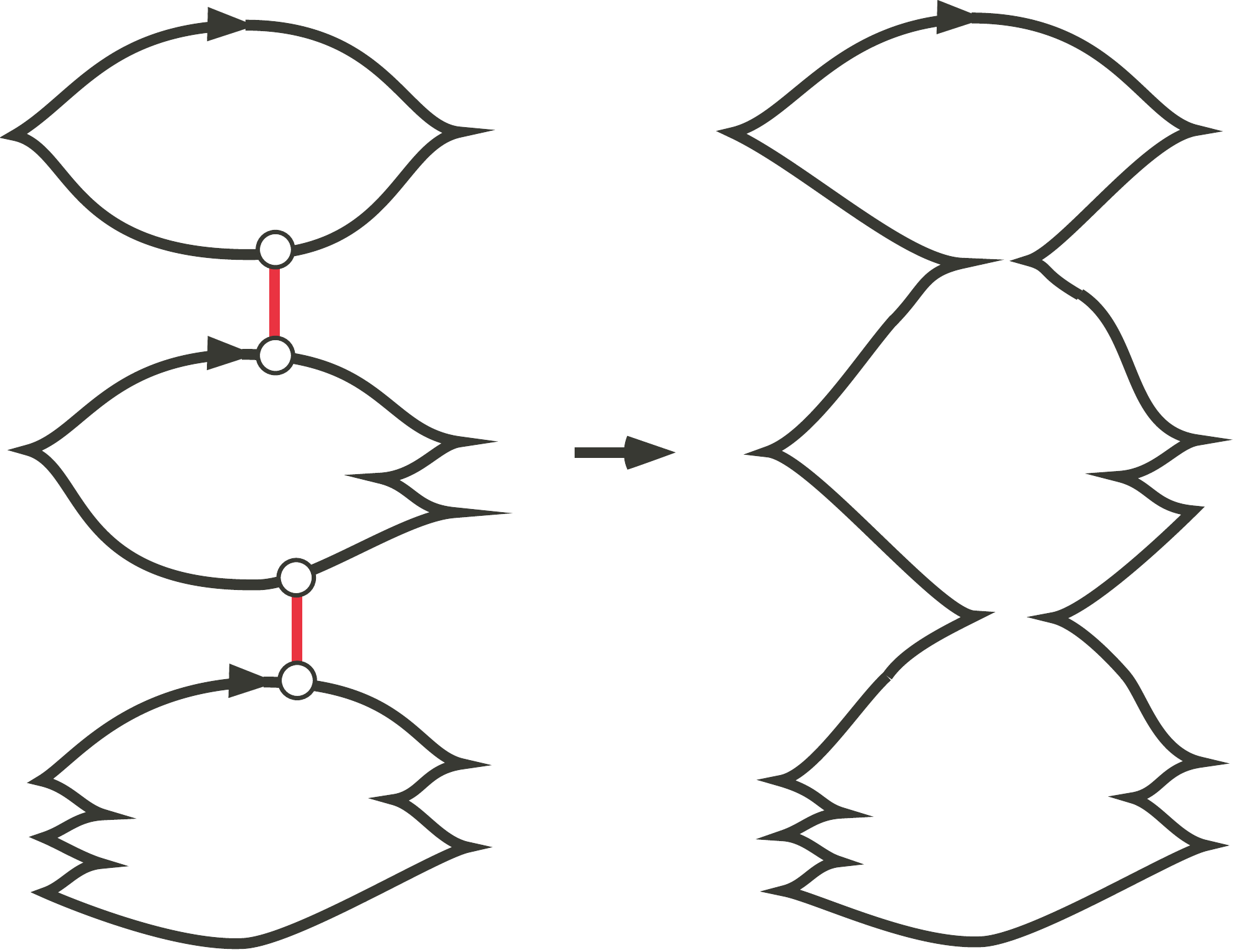}}
  \caption{After all crossings are eliminated, a Legendrian isotopy can be applied so that $\Lambda_0$ is a stack
  of $c$ Legendrian unknots oriented clockwise.  After $c-1$ additional surgeries, a Legendrian unknot is obtained.}
  \label{fig:unknot-stack}
\end{figure}

\begin{defn} \label{defn:sun} Given a Legendrian link $\Lambda$,  the (oriented)  {\bf Legendrian surgery unknotting number of $\Lambda$}, $\lso(\Lambda)$, is 
defined as the minimal length of an oriented, unknotting surgery string for $\Lambda$.
\end{defn}
 
 \begin{rem} \label{rem:basics} Here are some basic properties of $\lso(\Lambda)$:
\begin{enumerate}
\item  By Lemma~\ref{lem:parity+tb},  
for any Legendrian link $\Lambda$, the parity of $\lso(\Lambda)$ is opposite the parity of the number of components of $\Lambda$;
\item    
 For any oriented, Legendrian link $\Lambda$ with $j$ components, $j-1 \leq \lso(\Lambda) < \infty$, with $0 =  \lso(\Lambda)$ iff $\Lambda$ is topologically an unknot; 
 \item  If $\Lambda$ is a topolopologically non-trivial Legendrian knot and there exists an oriented unknotting surgery string for $\Lambda$
of length $2$,  then $\lso(\Lambda) = 2$;
\item If $\Lambda'$ is obtained from $\Lambda$ by stabilization(s), then $\lso(\Lambda') \leq \lso(\Lambda)$. 
\end{enumerate}
\end{rem}

Propositon~\ref{prop:unknottable} and, more importantly, 
explicit calculations will give  upper bounds for $\lso(\Lambda)$.
Now we turn to examining some lower bounds for $\lso(\Lambda)$.  

First, by Theorem~\ref{thm:unknot-classification},
 if $\Lambda'$ is a Legendrian unknot, then $tb(\Lambda') + |r(\Lambda')| \leq -1$.  Thus if $\Lambda$ is a Legendrian link with a 
 ``large"  Thurston-Bennequin and/or rotation number, one is forced to do a certain number of Legendrian surgeries.  More precisely, 
 Lemma~\ref{lem:parity+tb} implies:
 
 \begin{lem}\label{lem:tb-lb} For any Legendrian link $\Lambda$,
  $$tb(\Lambda) + |r(\Lambda)| + 1 \leq \lso(\Lambda).$$ 
\end{lem}

Lemma~\ref{lem:tb-lb} gives us improved lower bounds over those given in Remark~\ref{rem:basics} when  $2 \leq tb(\Lambda) + |r(\Lambda)|$. 
  \footnote{The parity of $tb(\Lambda) + |r(\Lambda)|$ agrees with the parity of the number of components of $\Lambda$, so for knots, we get
  interesting new bounds when $3 \leq tb(\Lambda) + |r(\Lambda)|$.}
For example, there exists a Legendrian whose underlying smooth knot type is $m(5_1)$  and whose classical
invariants satisfy  $tb(\Lambda)  +| r(\Lambda)|= 3$; see, for example, \cite{chongchitmate-ng}.  Thus 
 Lemma~\ref{lem:tb-lb} implies
$4 \leq \lso(\Lambda)$.  However 
for many links,  $tb(\Lambda) + |r(\Lambda)| \leq 2$.  For example, for any  Legendrian $\Lambda$ that is topologically the $5_1$ knot,  
 $tb(\Lambda) + |r(\Lambda)| \leq -5$.  Although Lemma~\ref{lem:tb-lb} will not help us, in this case we can make use of:

    \begin{lem}\label{lem:g4-lb}  For  a Legendrian link $\Lambda$ with $j$ components, let $L_\Lambda$ denote the underlying smooth 
    link type of $\Lambda$, 
  and let $g_4(L_\Lambda)$ denote the smooth $4$-ball genus of  $L_\Lambda$.  
  Then $$2 g_4(L_\Lambda) + (j-1) \leq \lso(\Lambda).$$
  \end{lem}
  
  \begin{proof}  From a Legendrian surgery string of length $n$ that ends at an unknot, one can construct 
  a smooth, orientable, compact,  and connected $2$-dimensional surface in $B^4$ with boundary equal to $L_\Lambda$
  and Euler characteristic equal to 
  $1 - n$;  the genus, $g$, of this surface satisfies $1 - n= 2 - 2g - j$.
  Thus, by definitions of  the smooth $4$-ball genus, 
  $$(j-1) + 2g_4(L_\Lambda) \leq (j-1) + 2g = n.$$
  Since $\lso(\Lambda)$ is the minimum length of a surgery unknotting string, the claim follows.
    \end{proof}

    A convenient table of smooth $4$-ball genera of knots can be found at Cha and Livingston's KnotInfo website, 
    \cite{knot-info}.

       \section{The Surgery Unknotting Number for Families of Knots}
  \label{sec:families}
 
 In this section we will calculate the surgery unknotting numbers for Legendrian twist knots, Legendrian torus links,
 and positive, Legendrian rational links.  The fact that we can precisely calculate these numbers for the first two families rests upon
 classification results of \cite{etnyre-ng-vertesi}, \cite{etnyre-honda:knots}, and \cite{dalton}.  
 
    \begin{subsection}{Legendrian twist knots}   
  
   A twist knot is a knot that is smoothly equivalent to a knot $K_m$  in the form 
   of  Figure~\ref{fig:top-twist}.  In other words, a twist knot is a twisted Whitehead double of the unknot.      
      
       \begin{figure}
  \centerline{\includegraphics[height=.7in]{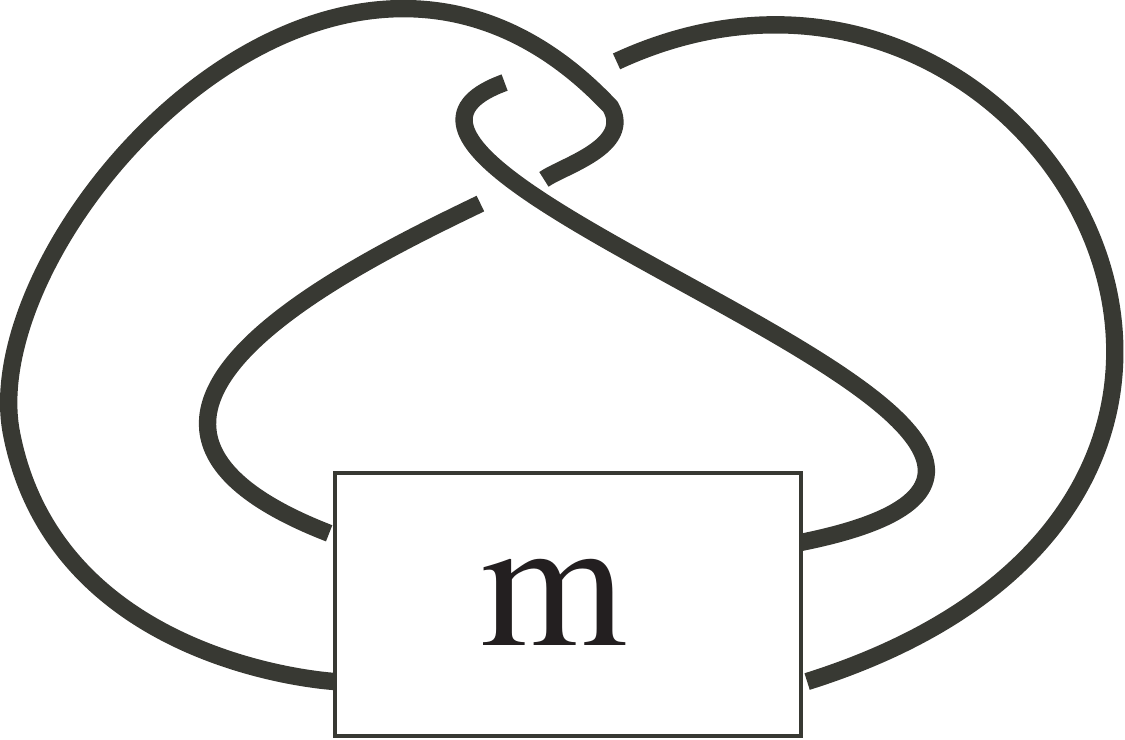}}
  \caption{The twist knot $K_m$; the box contains $m$ right-handed half twists if
  $m \geq 0$, and $|m|$ left-handed twists if $m < 0$. Notice that $K_0$ and $K_{-1}$ are unknots.}
  \label{fig:top-twist}
\end{figure}

  \begin{thm}\label{thm:twist}  If $\Lambda$ is a Legendrian knot that is topologically a non-trivial twist knot
  then  $\lso(\Lambda) = 2$. 
  \end{thm}
  
  \begin{proof}  
  Etnyre, Ng, and V\'ertesi,  have  classified  all
  Legendrian twist knots, \cite{etnyre-ng-vertesi}.  In particular, every Legendrian knot $\Lambda$ with maximal Thurston-Bennequin
   invariant that is topologically $K_m$, for some $m \leq -2$, is Legendrian isotopic to one of the form
in Figure~\ref{fig:gen-neg-twist}, and  every Legendrian knot $\Lambda$ with maximal Thurston-Bennequin
 invariant that is topologically $K_m$, for $m \geq 1$ with maximal Thurston-Bennequin invariant
 is Legendrian isotopic to one of the form in Figure~\ref{fig:gen-pos-twist}. \footnote{We omit $m = 0, -1$ since
those corresponds to the unknot.}  
 Every Legendrian knot $\Lambda$ that is topologically a non-trivial twist knot
is obtained by stabilization of one of these with maximal Thurston-Bennequin invariant.
By Remark~\ref{rem:basics},  it suffices to show for any  Legendrian knot $\Lambda^+$ that is topologically 
  a non-trivial twist knot and has maximal Thurston-Bennequin invariant, $\lso(\Lambda^+) = 2$.
For $\Lambda^+$,  we can do  the two unknotting surgeries near the ``clasp".  The sign of
  the crossings in the clasp will depend on whether $m$ is even or odd:    
  Figure~\ref{fig:twist-surg} shows the positions
  of two surgeries that result in an unknot.
  \end{proof}

   \begin{figure}
 \centerline{\includegraphics[height=.7in]{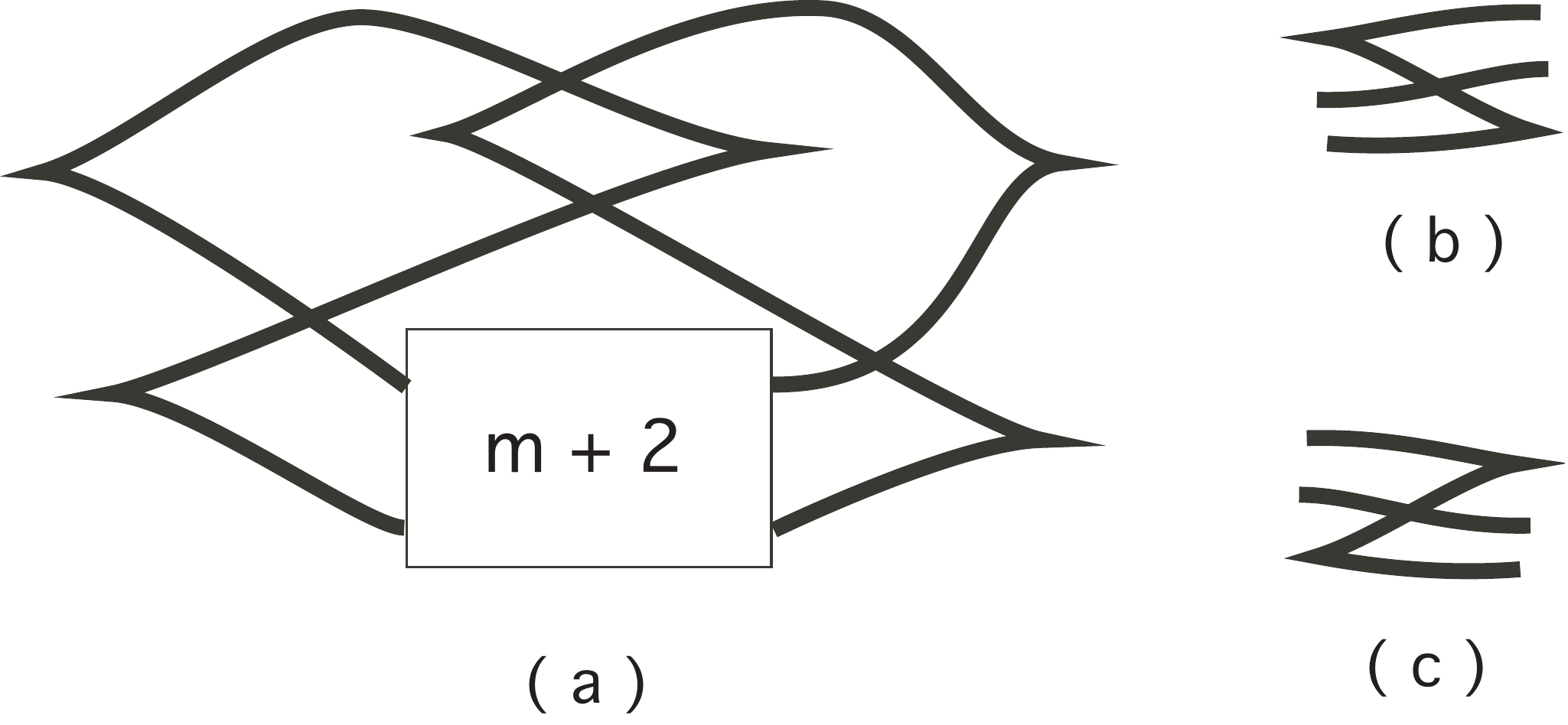}}
 \caption{Any Legendrian knot that is topologically a negative twist knot, $K_m$ with $m \leq -2$, and has maximal Thurston-Bennequin
  invariant 
  is Legendrian isotopic to one of the form in 
  (a) where the box contains $|m+2|$ half twists, each of form $S$ as shown in (b) or of form $Z$ as
  shown in (c).}
  \label{fig:gen-neg-twist}
\end{figure}
 \begin{figure}
 \centerline{\includegraphics[height=.7in]{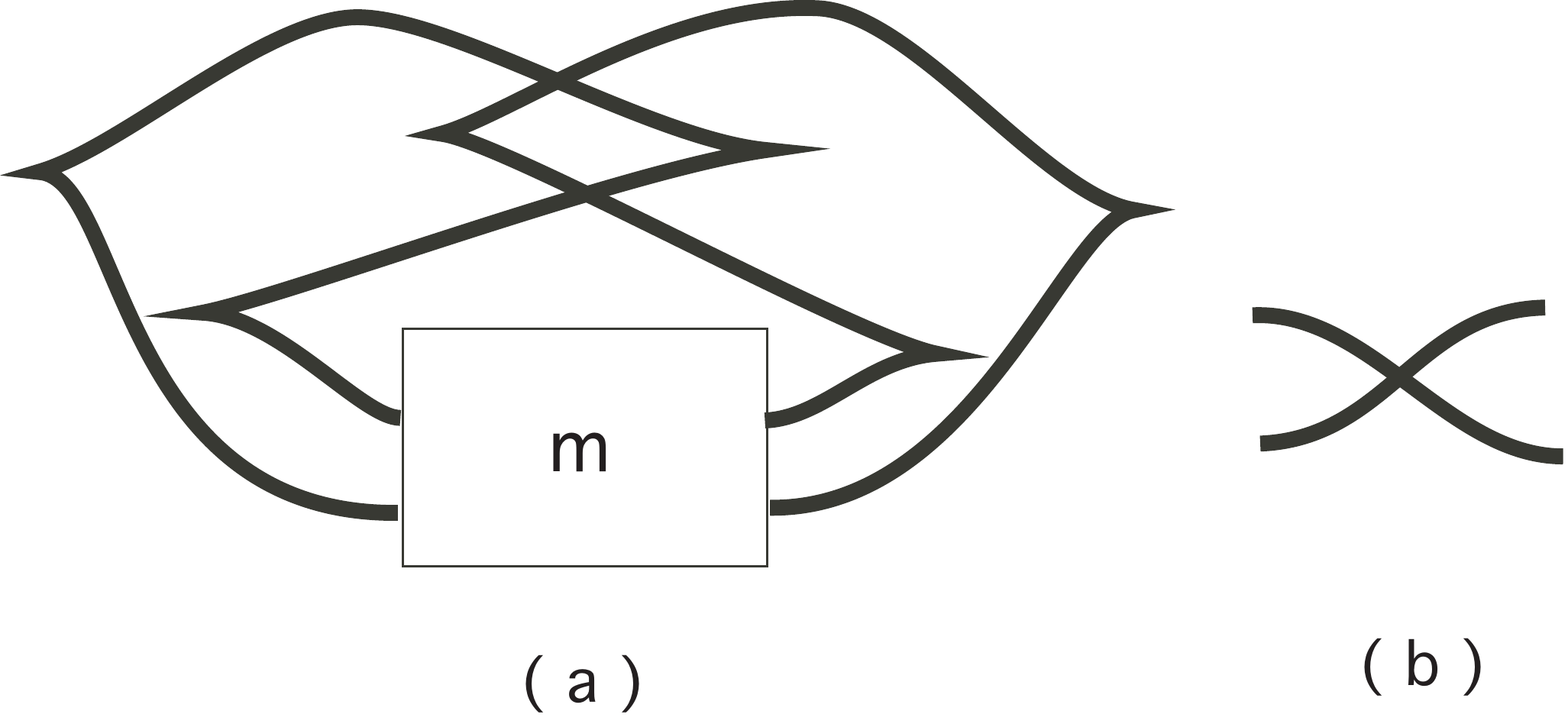}}
  \caption{Any Legendrian knot that is topologically a positive twist knot, $K_m$ with $m \geq 0$, and has maximal Thurston-Bennequin
  invariant 
  is Legendrian isotopic to one of the form in 
  (a) where the box contains $m$ half twists, each of form $X$ as shown in (b).}
   \label{fig:gen-pos-twist}
\end{figure}
   \begin{figure}
 \centerline{\includegraphics[height=1.8in]{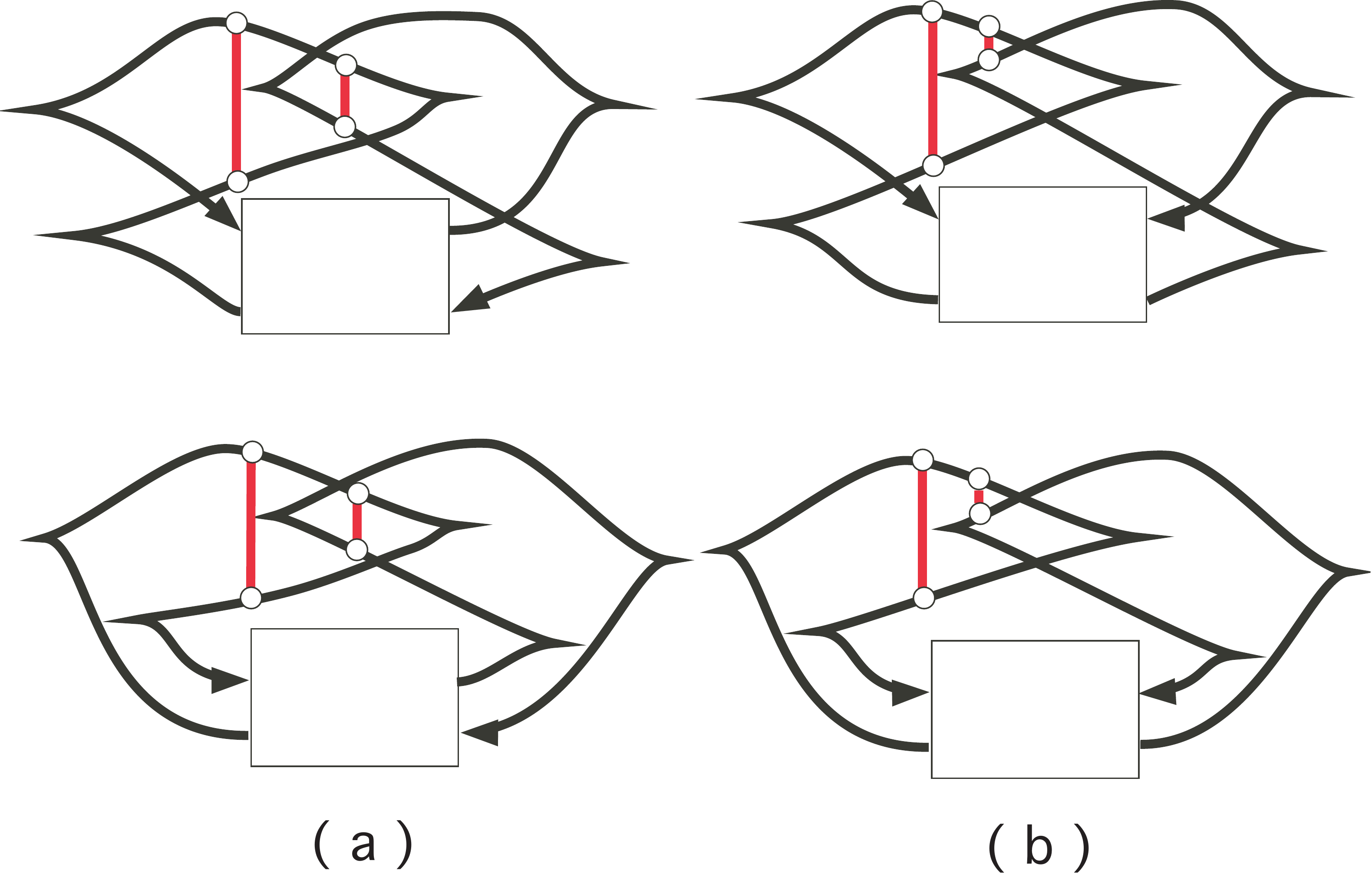}}
	  \caption{For a Legendrian knot with maximal Thurston-Bennequin invariant that is topologically $K_m$,
	  (a) gives the surgery points when $m$ is even, and  (b) gives the surgery points when $m$ is odd.}
  \label{fig:twist-surg}
\end{figure}

   \end{subsection}

 \begin{subsection}{Legendrian Torus Links}

A torus link is a link that can be smoothly isotoped so that it lies on the surface of an unknotted torus in $\rr^3$.  Every torus
knot can be specified by a pair $(p,q)$ of coprime integers:  we will use the convention that 
the $(p,q)$-torus knot, $T(p,q)$, winds $p$ times around a meridonal curve of the torus and $q$ times in the longitudinal direction.
See, for example, \cite{adams}.
In fact, $T(p, q)$ is equivalent to  $T(q, p)$ and to $T(-p, -q)$.  So we will
 always assume that $|p| > q > 0$; in addition we will assume $q > 1$ since we are interested in non-trivial torus knots.     
For $j\geq 2$,  $T(jp, jq)$, with $|p| > q > 1$ and $\gcd(p,q) = 1$, will be a $j$-component link where each component is
a $T(p,q)$ torus link.  We will only consider torus links of non-trivial components.

  \begin{thm}\label{thm:torus} %
  If $\Lambda$ is a $j$-component  Legendrian link that is topologically the $(jp, jq)$-torus link, $|p| > q > 1$,
  then $\lso(\Lambda) =  (|jp|-1)(jq-1)$.
  \end{thm}
  
  \begin{proof}  First consider the case where $\Lambda$  is topologically a positive torus knot, $T(p, q)$ with $p > 0$. As shown by
 Etnyre and Honda, \cite{etnyre-honda:knots},  
 the list of different Legendrian representations of a  positive torus knot can be represented as a ``tree" in parallel
 to the tree of unknots shown in Figure~\ref{fig:unknots}.  Namely, for 
 fixed $p > q>1$,  there is a unique Legendrian knot $\Lambda^+$ that is topologically $T(p,q)$ with
 maximal Thurston-Bennequin invariant $tb(\Lambda^+) = pq-p-q$ and $r (\Lambda^+)= 0$; any Legendrian knot $\Lambda$ that is topologically
 $T(p,q)$  is obtained
 by stabilizations of $\Lambda^+$.
   By Remark~\ref{rem:basics}, it suffices
  to  show that if $\Lambda^+$ is a Legendrian knot that is topologically
  $T(p,q)$ and has maximal Thurston-Bennequin invariant, then $\lso(\Lambda^+) = (p-1)(q-1)$.
  By Lemma~\ref{lem:tb-lb}, 
  $$tb(\Lambda^+) + |r(\Lambda^+)| + 1 = (p-1)(q-1) \leq \lso(\Lambda).
  $$ 
  In fact, it is possible to unknot with $(p-1)(q-1)$ surgeries.  Starting from the left most string of crossings,
  do $(q-1)$ successive surgeries as illustrated for the $(5,3)$-torus knot in Figure~\ref{fig:5-3-torus-unknot}.  In general,
  this take the $(p,q)$-torus knot to the $(p-1,q)$-torus link.  Repeating this $(p-1)$ times results in the
  $(1,q)$-torus knot, which is an unknot.
   \footnote{By Corollary~\ref{cor:realize-g4}, we can now reprove the Milnor conjecture as mentioned in Corollary~\ref{cor:Milnor}.}

  The above proof easily generalizes to positive torus links of non-trivial components.  Dalton showed in \cite{dalton} that 
  there is a unique Legendrian link $\Lambda^+$ that is topologically $T(jp,jq)$ with
 maximal Thurston-Bennequin invariant $tb(\Lambda^+) = jpjq-jp-jq$.  The construction of this
 one exactly parallels the construction in Figure~\ref{fig:5-3-torus-unknot}, and so  the same pattern of $(jp-1)(jq-1)$ surgeries of 
 will produce a Legendrian unknot.

   \begin{figure}
 \centerline{\includegraphics[height=1.5in]{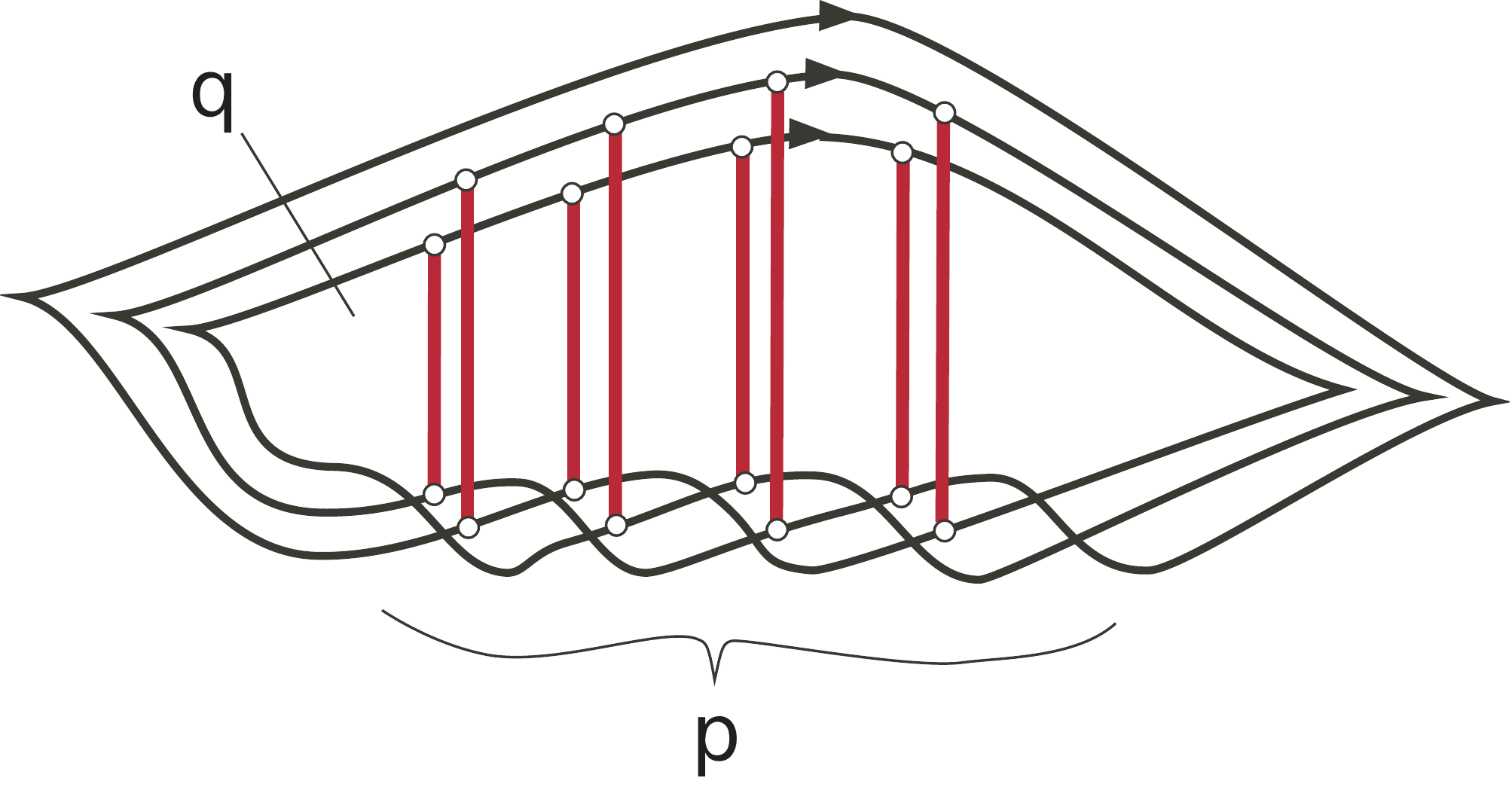}}
  \caption{The Legendrian $(5,3)$-torus knot with maximal tb invariant.  The general,
  positive Legendrian $(p,q)$-torus knot with maximal Thurston-Bennequin invariant is constructed using
  $q$ strands and a length $p$ string of crossings.  Shown are the $(p-1)(q-1)$ oriented Legendrian surgeries that unknot the Legendrian positive
  $(p,q)$-torus knot with maximal tb.}
  \label{fig:5-3-torus-unknot}
\end{figure}

 Next consider the case where $\Lambda$ is topologically a negative torus knot, $T(p, q)$ with $p < 0$. In this case,
 Etnyre and Honda have shown that 
the list of different Legendrian representations of a  negative torus knots, $T(p, q)$ for $p < 0$ and $|p| > q > 1$,
    can be represented as a ``mountain range" where the number of representatives with maximal Thurston-Bennequin
    invariant depends on the divisibility of $p$ by $q$.  Namely,  if
    $|p| = m q + e$, $0 < e < q$,      then  
    there will be $2m$
     Legendrian representatives of $T(p, q)$ 
     with maximal Thurston-Bennequin
    invariant of $pq < 0$.   Half of these different representatives  with maximal
    Thurston-Bennequin invariant are obtained by writing $m =  1 + n_1 + n_2$,  where $ n_1, n_2 \geq 0$,
    and  then $\Lambda_{(n_1, n_2)}^+$ is constructed using the form shown in Figure~\ref{fig:neg-torus-unknot} with $n_1$ and $n_2$ copies of the tangle
    $B$ inserted as indicated:  $r(\Lambda_{(n_1, n_2)}^+)  = q (n_2 - n_1) + e$.
    The other $m$ Legendrian versions of $T(p,q)$ with maximal Thurston-Bennequin invariant 
     are obtained by reversing the orientation.  For negative torus knots, Lemma~\ref{lem:tb-lb} will not be a useful lower bound.  
    However, since the calculation of the $4$-ball genus is the same for both the knot and its mirror,
    the calculations in the positive torus knot case 
    and Corolloary~\ref{cor:realize-g4}, (or \cite{kronheimer-mrowka}), show that for a negative torus knot $T(p,q)$,  $2g_4(T(p, q)) =  (|p| - 1)(q-1)$.  Thus, by Lemma~\ref{lem:g4-lb}
  $$(|p| - 1)(q-1) \leq \lso(\Lambda).$$ In fact, it is possible to arrive at an unknot with $(|p| - 1)(q-1)$ surgeries. 
   Figure~\ref{fig:neg-torus-unknot} shows the claimed surgeries: a surgery is done to the right of all crossings
   in the $L$, $R$, and $B$ regions (contributing $\frac12 q(q-1) +  \frac12 q(q-1) + (n_1 + n_2)q(q-1)$ surgeries), 
 and between each ``$Z$" in the $e$ string one does $(q-1)$ successive surgeries (contributing $(e-1)(q-1)$ surgeries).  Thus the total
 number of surgeries is:
 $$(1 + n_1 + n_2) q (q-1) + (e-1)(q-1) =  (mq + e - 1)(q-1) = (|p| - 1)(q-1).$$
    \begin{figure}
 \centerline{\includegraphics[height=2in]{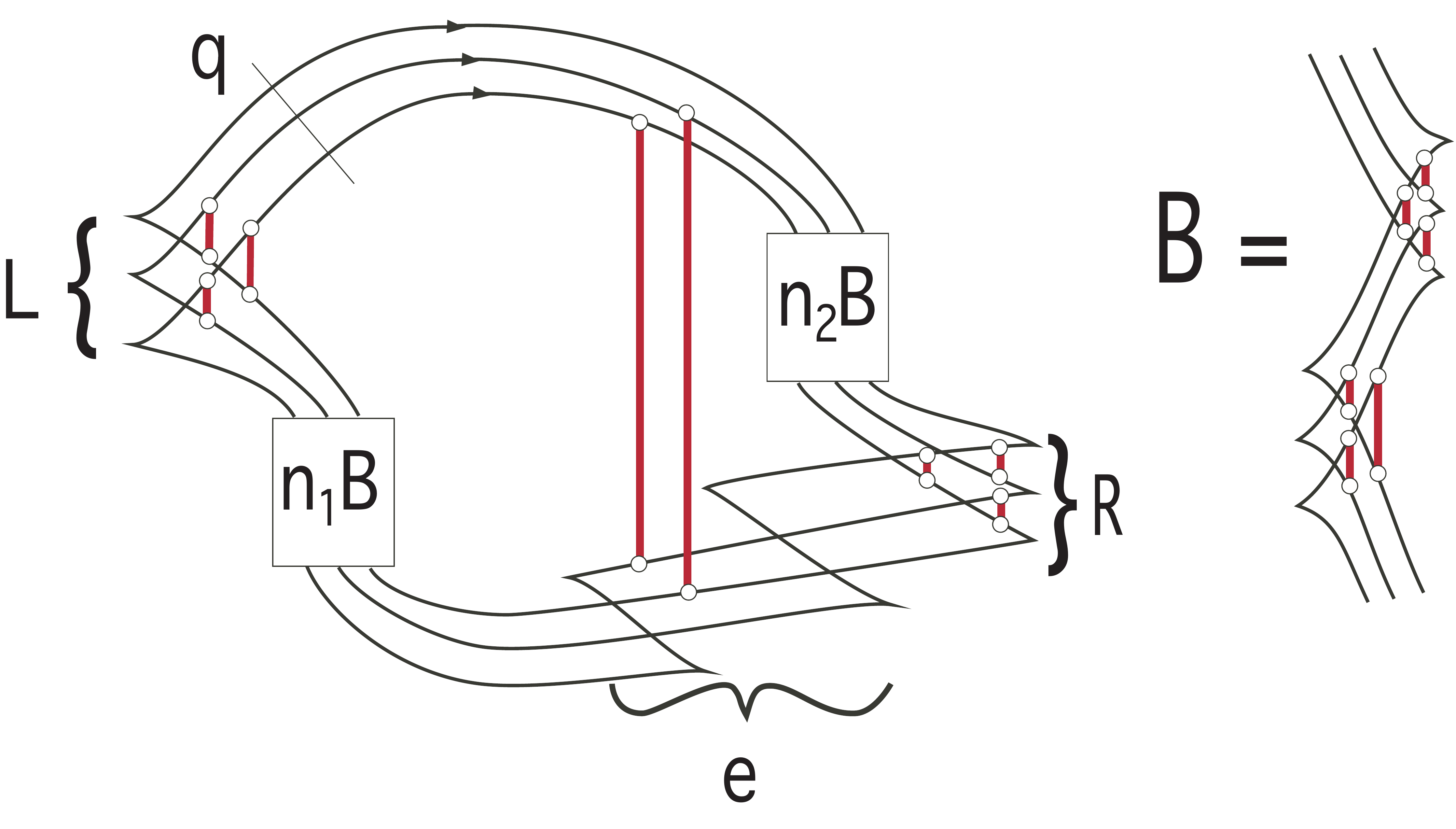}}
  \caption{The $(|p|-1)(q-1)$ oriented Legendrian surgeries that unknot a Legendrian negative
  $(p,q)$-torus knot with maximal Thurston-Bennequin invariant.}
  \label{fig:neg-torus-unknot}
\end{figure}
  
  The above proof easily generalizes to negative torus links.  It follows from \cite{nakamura} that 
  $g_4(T(jp, jq)) + (j-1) =  (j|p|-1)(jq-1)$; see Remark~\ref{rem:g4-torus-links}. Dalton showed in \cite{dalton} that 
  there are $2m$ Legendrian links $\Lambda^+$ that are topologically $T(jp,jq)$ with
 maximal Thurston-Bennequin invariant, and all Legendrians that are topologically $T(jp, jq)$ are obtained by
 stabilizations of one of these.    Each of these with maximal Thurston-Bennequin invariant  
 can be constructed as in Figure~\ref{fig:neg-torus-unknot}, 
 and so  the same pattern of $(j|p|-1)(jq-1)$ surgeries 
 will produce a Legendrian unknot.
 \end{proof}
  
  \begin{rem} \label{rem:g4-torus-links} Nakamura's formula, \cite{nakamura}, for the smooth $4$-ball genus of a $j$-component positive link $L$ is that 
  $$2 g_4(L) = 2 - j - s(D) + c(D),$$
  where $s(D)$ is the number of Seifert circles and $c(D)$ is the number of crossings in a non-split positive diagram $D$ for $L$.
  It is straightforward to see that when $L$ is the positive torus link $T(jp, jq)$, using the diagram $D$ corresponding to Figure~\ref{fig:5-3-torus-unknot},
  $s(D) = jq$ and $c(D) = jp(jq-1)$.  So, 
  $$2 g_4(T(jp, jq)) = 2 - j - jq + jp(jq-1) = (1-j) + (jp-1) (jq-1).$$
  Thus for any Legendrian link $\Lambda$ that is topologically $T(jp, jq)$, for either $p$ positive or negative, 
  $$2g_4(T(jp, jq)) + (j-1) = \sun(\Lambda).$$
  \end{rem}
 
 \end{subsection}

  \begin{subsection}{Positive, Legendrian Rational Links}  
  
  \begin{defn}\label{defn:pos-rat} Given a vector of integers $(c_n, \dots, c_2, c_1)$, where $c_n \geq 2$,  
  and $n \geq 2$ implies $c_i \geq 1$ for $i = 1, \dots, n-1$,
     we construct the {\bf rational Legendrian link}
  $\Lambda(c_n, \dots, c_2, c_1)$ to be the Legendrian numerator closure of the Legendrian tangle
  $(c_n, \dots, c_2, c_1)$ as demonstrated in Figure~\ref{fig:rat-gen}; 
  see also \cite{adams}, \cite{traynor:strat}, \cite{schneider}.     The rational Legendrian link $\Lambda(c_n, \dots, c_2, c_1)$ is {\bf positive} if all crossings are positive.  
     
     \end{defn}
     
     This Legendrian link $\Lambda(c_n, \dots, c_1)$ is topologically the numerator closure of the rational tangle associated to the rational number $q$  
   with continued fraction expansion $q = c_1 + 1/(c_2 + 1/(c_3 + \dots ))$; \cite{conway}.

       \begin{figure}
 \centerline{\includegraphics[height=2in]{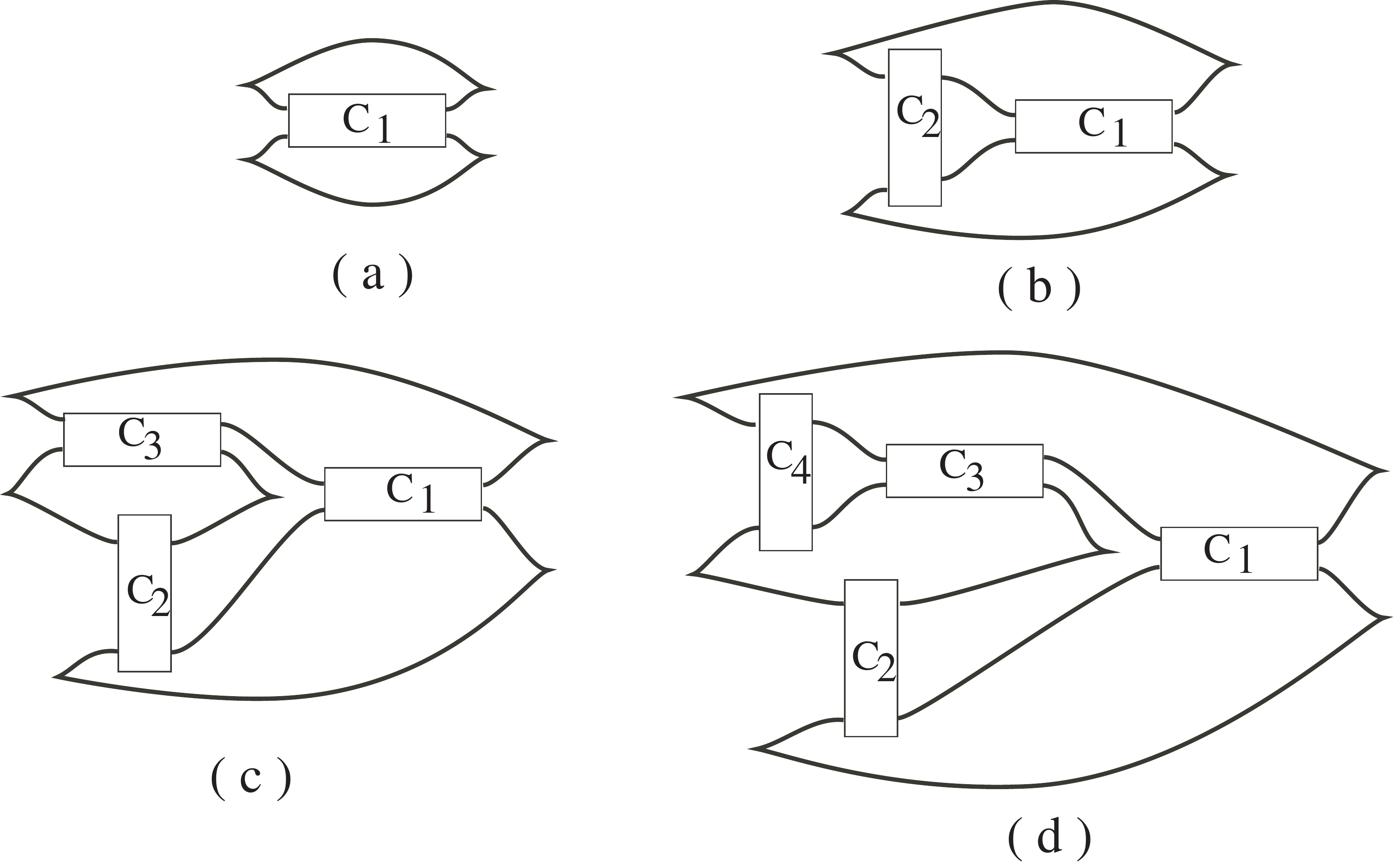}}
  \caption{(a) The general form of  $\Lambda(c_0)$, (b) the general form of $\Lambda(c_1, c_0)$, (c) the general form of $\Lambda(c_2, c_1, c_0)$,
  the general form of $\Lambda(c_3, c_2, c_1, c_0)$.}
  \label{fig:rat-gen}
\end{figure}

The ``even" entries  $c_2, c_4, \dots$ of the vector $(c_n, \dots, c_2, c_1)$  denote the strings of vertical crossings.  It is straightforward
to verify that  the parity of these
vertical entries determine when $\Lambda(c_n, \dots, c_1)$ is a positive link:

\begin{lem}   \label{lem:pos}
\begin{enumerate}
\item When $n$ is odd, there exists an orientation on the components of $\Lambda(c_n, \dots, c_1)$ so that it
is a positive link if and only if  
 $c_i$ is even,  for all $i$ even. 
Moreover, 
$\Lambda(c_n, \dots, c_1)$ is a knot when $\sum_{i \text{ odd}} c_i$ is odd.
\item When $n$ is even, there exists an orientation on the components of $\Lambda(c_n, \dots, c_1)$ so  it
is a positive link if and only if  $c_n$ is odd and $c_{n-2}, c_{n-4}, \dots, c_2$ are  all even. 
Moreover, 
 $\Lambda(c_n, \dots, c_1)$ is a knot when $\sum_{i \text{ odd}} c_i$ is even. 
\end{enumerate}
\end{lem}

The Legendrian surgery unknotting number of a positive link has a convenient formula in terms of the
``odd" entries, which correspond to the strings of horizontal crossings.  
There will be some differences in following formulas depending on whether $\Lambda$ is constructed from
an odd or an even length vector.  Define
$$p(n) = 
\begin{cases}
1, & n \text{ odd} \\
0, & n \text{ even};
\end{cases}
$$
$p(n)$ measures the parity of the ``length" of the vector $(c_n, \dots, c_1)$.
 
\begin{thm} \label{thm:pos-rat} If $\Lambda(c_n, \dots, c_2, c_1)$  is a positive, Legendrian rational link, then
$$\lso(\Lambda(c_n, \dots, c_2, c_1)) = \sum_{i \text{ odd }} c_{i}  - p(n).$$
\end{thm}

\begin{proof}   This will be proved using the lower bound on $\lso(\Lambda)$ provided by
Lemma~\ref{lem:tb-lb}, and explicit calculations.  

We will first show that 
$$r(\Lambda(c_n, \dots, c_2, c_1)) = 0, \text{ and }
tb(\Lambda(c_n, \dots, c_2, c_1)) =  \sum_{i \text{ odd}} c_{i}- p(n) - 1.$$
 It is easy to verify that when all the crossings are positive, the up and down cusps cancel in pairs and thus
the rotation number vanishes.  To calculate $tb(\Lambda(c_n, \dots, c_1))$, notice that when $n$ is odd 
 the number of right cusps 
is $2$ more than  the number of vertical crossings, $\sum_{i \text{ even}} c_i$, while when $n$ is even, the number of rights cusps
is $1$ more than the number of vertical crossings.  Thus:
$$tb(\Lambda(c_n, \dots, c_2, c_1)) = \sum_{i=1}^n c_i - \left( \sum_{i \text{ even}} c_i + 1  + p(n) \right) = \sum_{i \text{ odd}} c_i - 1 - p(n). $$
 Thus, by   Lemma~\ref{lem:tb-lb}, 
$$ \sum_{i \text{ odd}} c_{i}  - p(n) \leq \lso(\Lambda(c_n, \dots, c_2, c_1)).$$

In fact, it is possible to unknot $\Lambda(c_n, \dots, c_2, c_1)$  by doing
$c_i - 1$ surgeries in each horizontal component   
and $1$ surgery in each vertical segment;  Figure~\ref{fig:pos-rat-unknot} illustrates some examples of this.
When $n=1$, there are no vertical segments; for other odd $n$,  the number of vertical components is one less than the number
of horizontal components, and when $n$ is even,
 the number of vertical components agrees with the number
of horizontal components.   Thus
$$\lso(\Lambda(c_n, \dots, c_1)) \leq \sum_{i \text{ odd}} c_i  - p(n),$$ and  the desired
calculation of $\lso(\Lambda(c_n, \dots, c_1))$ follows.   
 \begin{figure}
 \centerline{\includegraphics[height=1.4in]{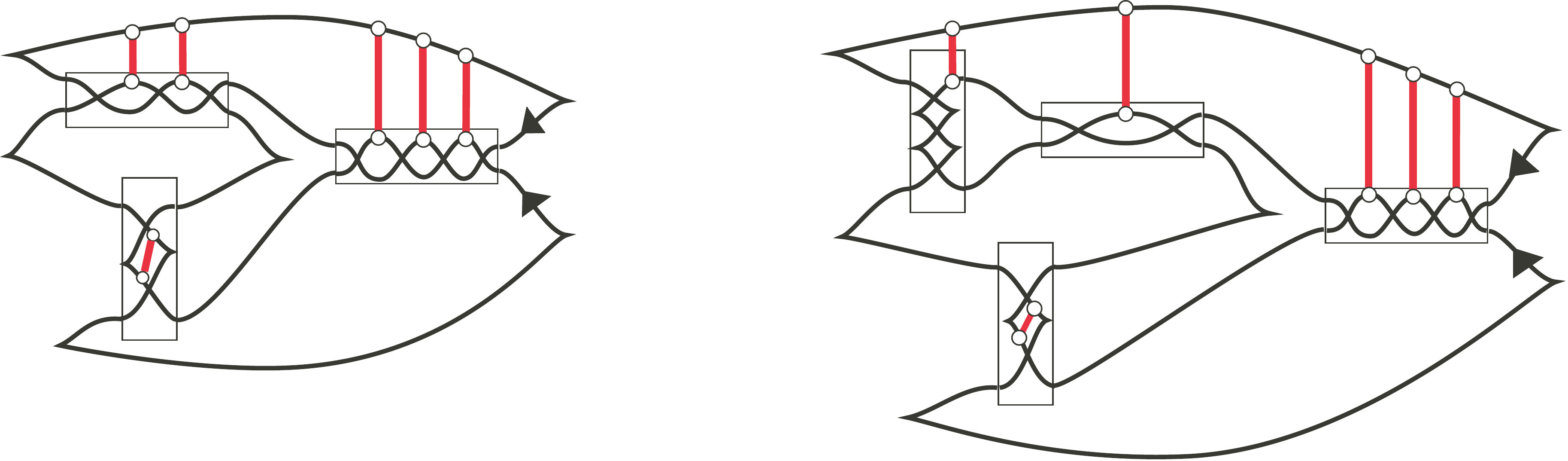}}
  \caption{Two positive, Legendrian rational knots of odd and even lengths.  In both cases, it is possible to
  unknot by doing $c_i - 1$ surgeries in each horizontal segment ($i$ odd) and $1$ surgery in
  each vertical segment.}
  \label{fig:pos-rat-unknot}
\end{figure}

\end{proof}

\begin{rem} \label{rem:g4-rats}  In the above proof,  $\sun(\Lambda(c_n, \dots, c_1))$ is obtained by realizing
the lower bound given by the classical Legendrian invariants.  Thus,  by
 Corollary~\ref{cor:realize-g4}, we see that when $\Lambda(c_n, \dots, c_1)$ has an underlying
 topological type of the knot $K_\Lambda$,
 $\sun(\Lambda(c_n, \dots, c_1)) = 2 g_4(K_\Lambda)$.   Moreover, when $\Lambda(c_n, \dots, c_1)$ has an underlying
 topological type of a $2$-component link $L_\Lambda$, we can compare $\sun(\Lambda(c_n, \dots, c_1))$
 to the smooth $4$-ball genus of $L_\Lambda$ using Nakamura's formula (see Remark~\ref{rem:g4-torus-links})  for the smooth $4$-ball
 genus of a positive link .
 When $n$ is odd, the number of Seifert circles is $s(D) = 2 + \sum_{i \text{ even}} c_i$, while when
 $n$ is even, $s(D) = 1 + \sum_{i \text{ even}} c_i$.  Thus we find that for a $2$-component, positive, Legendrian rational link $\Lambda(c_n, \dots, c_1)$,
 $$2g_4(L_\Lambda) + 1 = c(D) - s(D) + 1 = \sum_{i \text{ odd}} c_i - p(n) = \sigma_0(\Lambda(c_n, \dots, c_1)).$$
\end{rem}

 \end{subsection}

 \section{The Surgery Unknotting Number for Small Crossing Knots}

 Given the calculations of the previous section, it is natural to 
ask Queston~\ref{ques:sun-g4} in the Introduction.
 To investigate the knot portion of this question, we examined Legendrian representatives of low-crossing knots.      
 There is not  a Legendrian classification of all these knot types, but a conjectured
  classification of these knot types can be found in \cite{chongchitmate-ng}.  
 In the following, we  prove Proposition~\ref{prop:small-cross}, which says that 
 the surgery unknotting number  of the Legendrian agrees with twice the smooth $4$-ball genus
 of the underlying smooth knot for all  Legendrians that are topologically
 a non-slice knot with crossing number at most $7$.

 In Section~\ref{sec:families}, Proposition~\ref{prop:small-cross} is verified for all torus and twist knots.
 The only non-torus and non-twist
 knots with $7$ or fewer crossings are $6_2, m(6_2)$, $6_3 = m(6_3)$, $7_3$, $m(7_3)$, $7_4$, $m(7_4)$, $7_5$, $m(7_5)$,
 $7_6$, $m(7_6)$, $7_7$, and $m(7_7)$.    The needed calculations fall into three categories as described below.
  
 \begin{ex}  For the smooth knots $7_3, m(7_3), 7_4, m(7_4), 7_5$, and $m(7_5)$, the genus, $g_3$, agrees with
 the smooth $4$-ball genus $g_4$.  \footnote{This is also the situation for the torus and non-slice twist knots studied in Section~\ref{sec:families}.}
 In general, we find that for a Legendrian $\Lambda$ 
 whose underlying knot type $K_\Lambda$ 
 satisfied $g_3(K_\Lambda) = g_4(K_\Lambda)$,
  it is fairly straight forward to show that 
 $\sigma_0(\Lambda) = 2 g_4 (L_\Lambda)$.   For example, Figure~\ref{fig:g3-g4} shows all
 conjectured  representatives  of $7_3$, $m(7_3)$, $7_4$, $m(7_4)$, $7_5$, and $m(7_5)$ with
 maximal Thurston-Bennequin invariant (after perhaps selecting alternate orientations and/or
 performing the Legendrian mirror operation).  For each of these with maximal Thurston-Bennequin invariant,
 it is possible to unknot with  $2 g_4 (K_\Lambda)$ surgeries as indicated.  
  \begin{figure}
  \centerline{\includegraphics[height=2.4in]{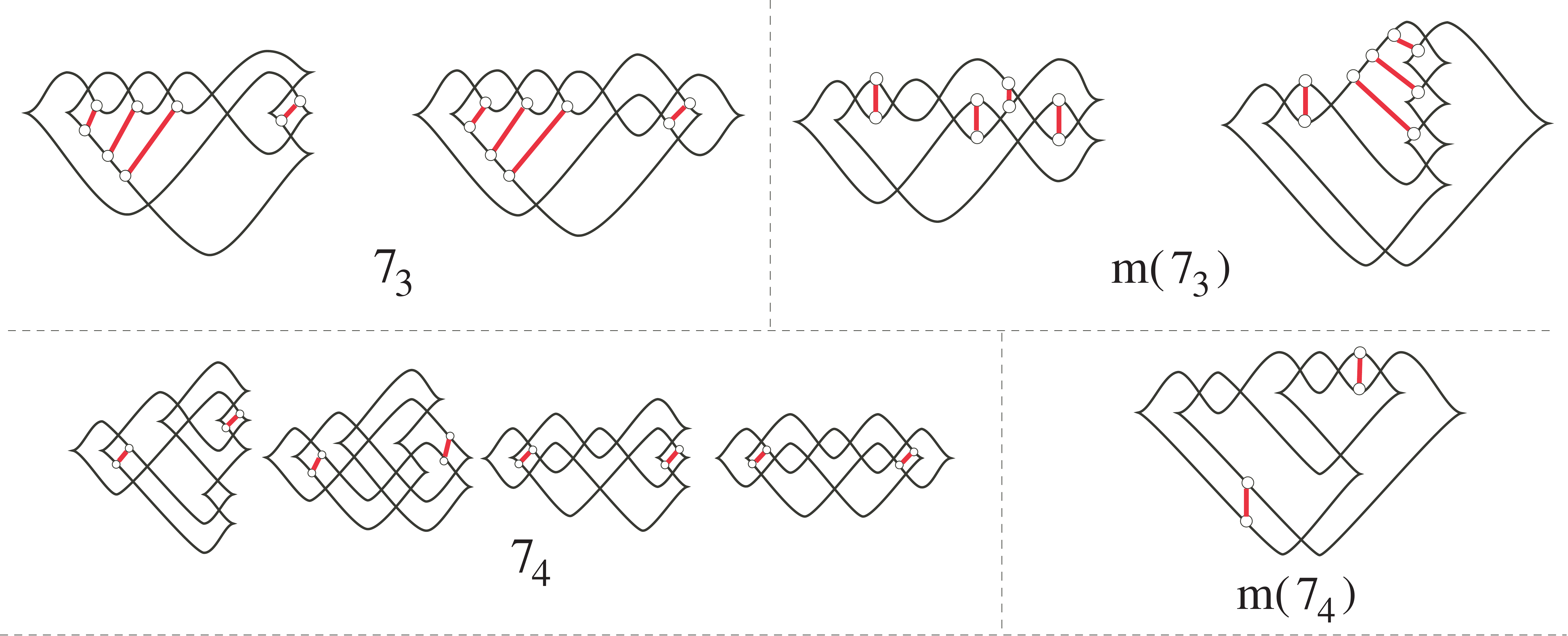}}
  \caption{Front projections representing all  conjectured Legendrian representatives of  $7_3$, $m(7_3)$, $7_4$, and $m(7_4)$ with maximal Thurston-Bennequin invariant.  For all of these knot types, $g_3(K_\Lambda) = g_4(K_\Lambda)$; the indicated surgery points realize $\lso(\Lambda) = 2 g_4(K_\Lambda)$.}
  \label{fig:g3-g4}
\end{figure} 
 \end{ex}

   In general, we found that for a Legendrian $\Lambda$ 
 whose underlying knot type $K_\Lambda$ 
 satisfied $g_4(K_\Lambda) < g_3(K_\Lambda)$, it is more difficult to calculate
 $\sigma_0(\Lambda)$.   To do calculations for our remaining cases, we made use of the well-known fact that the unknotting number of a knot, $u(K)$, gives
 an upper bound to the smooth $4$-ball genus:
 \begin{equation} \label{ineq:g4-u}
 g_4(K) \leq u(K).
 \end{equation}
 Figure~\ref{fig:g4-u} demonstrates two topological surgeries that produce a crossing change; an argument as in the proof of
 Lemma~\ref{lem:g4-lb} then proves Inequaltiy~\ref{ineq:g4-u}.  Notice that the topological Reidemeister moves used in the
 equivalence are not Legendrian Reidemeister moves.  However, near a negative crossing, it is possible to
 ``Legendrify" this construction:
  
    \begin{figure}
 \centerline{\includegraphics[height=.5in]{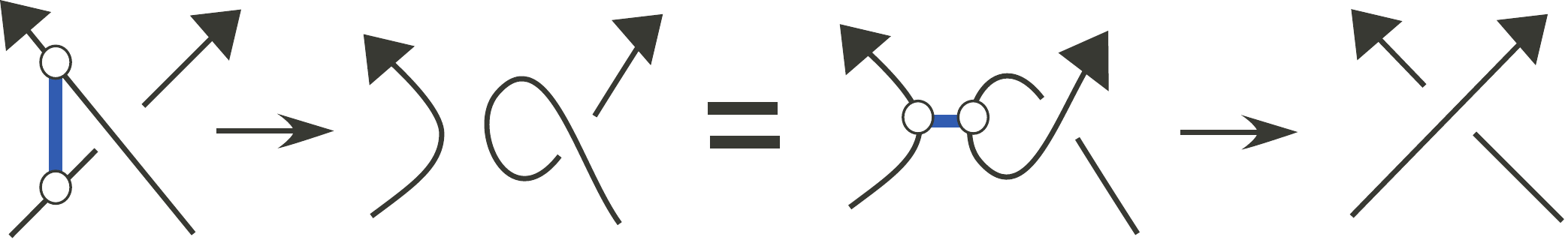}}
  \caption{ A sequence of two topological surgeries in a neighborhood of a negative crossing that 
  toplogically change the crossing. An analogous picture shows that a positive crossing can be changed into
  a negative crossing by two topological surgeries.}
  \label{fig:g4-u}
\end{figure}

   \begin{lem} \label{lem:neg-cross-unknot}  If the Legendrian knot $\Lambda$ has a front projection  that can
   be  topologically unknotted by changing a negative crossing, then 
   $$\lso(\Lambda) \leq 2.$$
   \end{lem}
   
   \begin{proof}  Figure~\ref{fig:neg-unknot} demonstrates how two surgeries can locally produce a topological crossing
   change.  
   
    \begin{figure}
 \centerline{\includegraphics[height=1.4in]{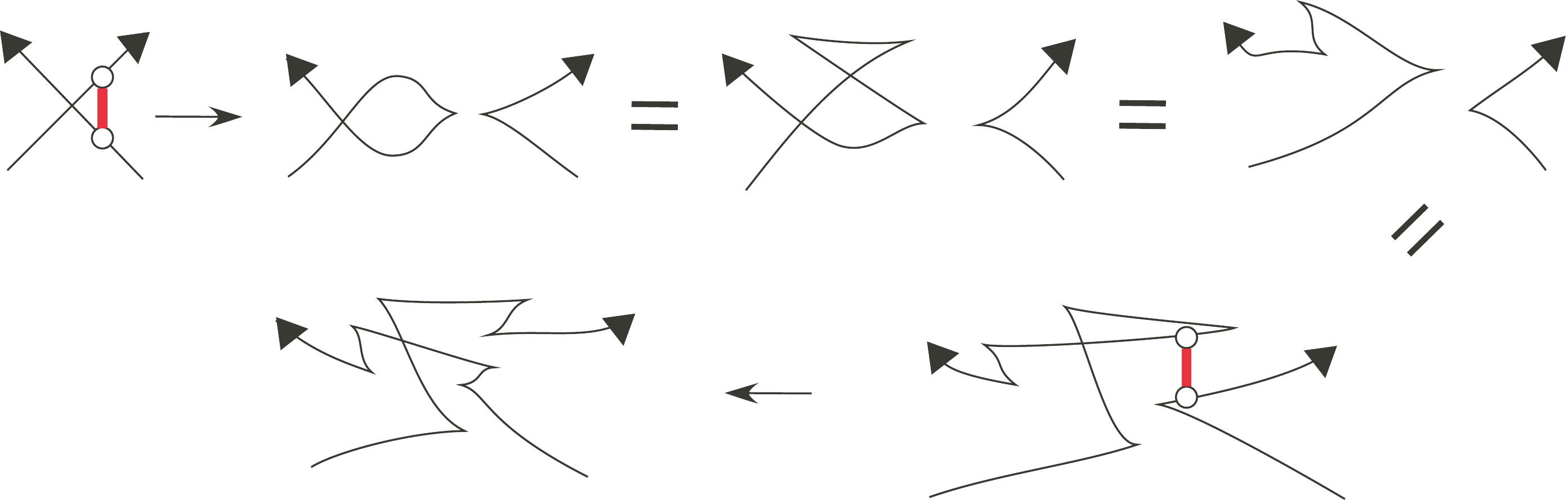}}
  \caption{ A sequence of two oriented surgeries in a neighborhood of a negative crossing that 
  toplogically change the crossing.}
  \label{fig:neg-unknot}
\end{figure}

\end{proof}

 \begin{ex}   Using Lemma~\ref{lem:neg-cross-unknot}, it is possible to show that for any conjectured Legendrian representative $\Lambda$
 of $6_2$, $6_3$, $7_6$, or $7_7$, $\sigma_0(\Lambda) = 2 g_4(L_\Lambda)$.  Figure~\ref{fig:neg-pt-unknot} shows the conjectured
 Legendrian representatives of these knot types with maximal Thurston-Bennequin invariant
 (after perhaps selecting alternate orientations and/or
 performing a mirror operation) and the negative crossing that when topologically changed  produces an unknot.
  \begin{figure}
 \centerline{\includegraphics[height=2in]{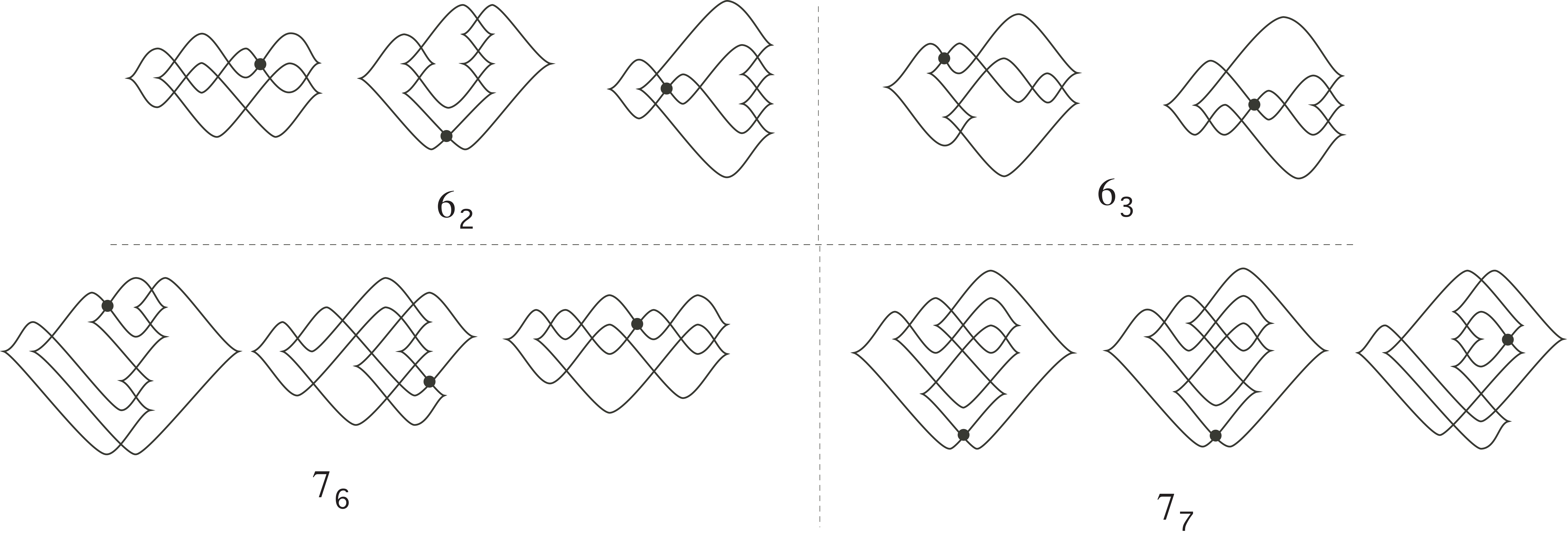}}
  \caption{ Front projections representing all conjectured Legendrian representatives of  $6_2$, $6_3$, $7_6$, and $7_7$ with maximal Thurston-Bennequin invariant.  These projections  can be topologically unknotted at the indicated negative crossing.}
  \label{fig:neg-pt-unknot}
\end{figure} 
 
 \end{ex}

   We were not able to find front projections of the conjectured maximal Thurston-Bennequin representatives
   of $m(6_2)$, $m(7_6)$, or $m(7_7)$ that could be topologically unknotted by changing a negative crossing;
   in fact, by  \cite{signed-unknot},  it is {\it not} possible to do this even in the smooth setting.
   Luckily, sometimes we can topologically change a positive crossing when it has a special form.

   \begin{defn}  A positive crossing is of {\bf S form}, {\bf Z form}, or {\bf hooked-X form} if it takes the form 
   as shown  in Figure~\ref{fig:pos-S-Z}.
    \end{defn}
   
     \begin{figure}
 \centerline{\includegraphics[height=.8in]{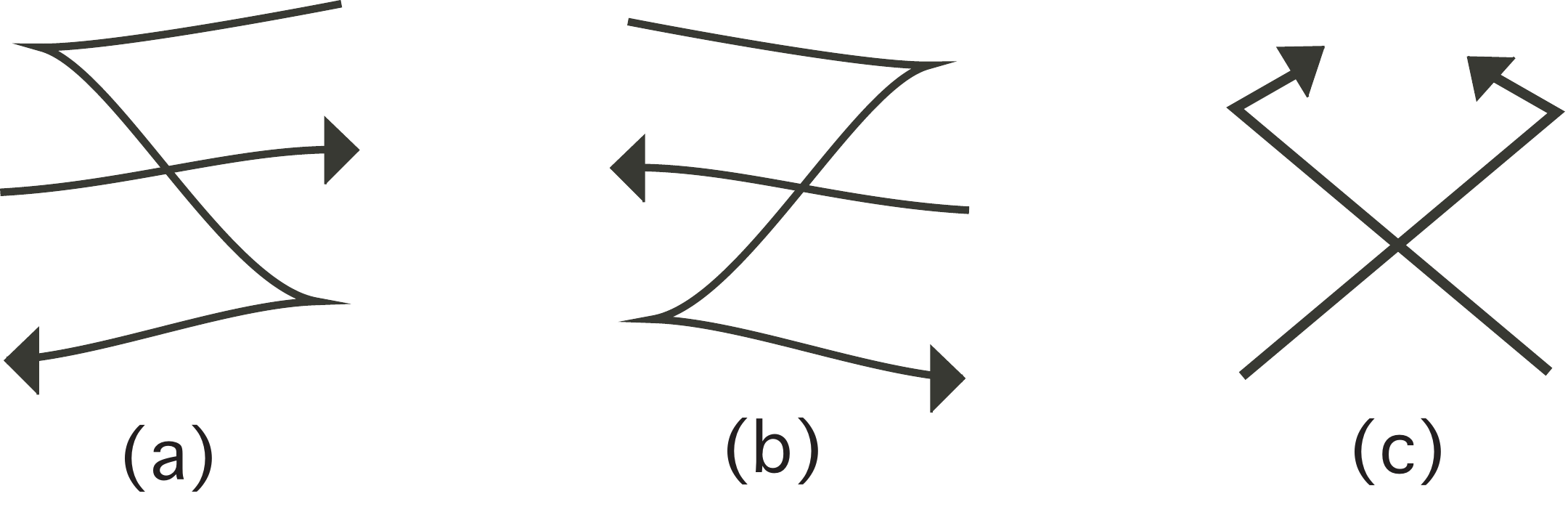}}
  \caption{A positive crossing of (a)  S form, (b) Z form, and (c) Hooked-X form.  Reversing the orientations
  on both strands keeps the respective forms.  Also reflecting the planar figure in (c) about a horizontal line produces
  another Hooked-X form.}
  \label{fig:pos-S-Z}
\end{figure}

   \begin{figure}
 \centerline{\includegraphics[width=3in]{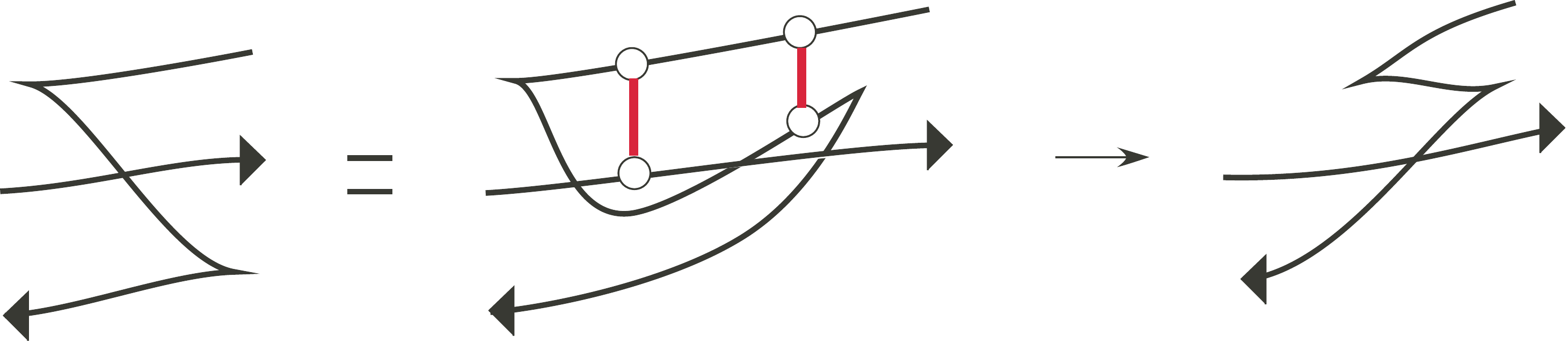}}
  \caption{ A positive crossing of S form can be transformed into a negative crossing
  with $2$ surgeries.  Similarly, a positive crossing of $Z$ form can be transformed into a negative
  crossing with $2$ surgeries.}
  \label{fig:pos-cross-unknot}
\end{figure}

 \begin{figure}
 \centerline{\includegraphics[width=5in]{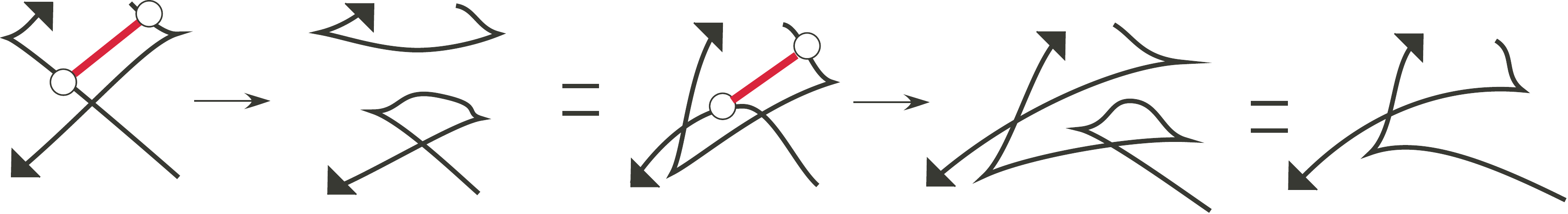}}
  \caption{A positive crossing of hooked-X form can be transformed into a negative crossing with
  $2$ Legendrian surgeries. }
  \label{fig:hook-X}
\end{figure}

   \begin{lem}\label{lem:pos-cross-unknot} If $\Lambda$ is a non-trivial Legendrian knot that has a projection that can be topologically unknotted by
   changing a positive crossing in  S, Z, or hooked-X form, then 
   $$\lso(\Lambda) \leq 2.$$ 
   \end{lem}

   \begin{ex}  Using Lemma~\ref{lem:pos-cross-unknot}, it is possible to show that for any conjectured Legendrian representative $\Lambda$
 of $m(6_2)$, $m(7_6)$, or $m(7_7)$, $\sigma_0(\Lambda) = 2 g_4(K_\Lambda)$.  Figure~\ref{fig:pos-pt-unknot} shows the conjectured
 Legendrian representatives of these knot types with maximal Thurston-Bennequin invariant (after perhaps selecting alternate orientations and/or
 performing a mirror operation).  These projections differ from those in \cite{chongchitmate-ng} by Legendrian Reidemeister moves of type II and III.
 The black dot indicates a positive
  crossing that when topologically changed  produces an unknot.

  \begin{figure}
 \centerline{\includegraphics[width=5in]{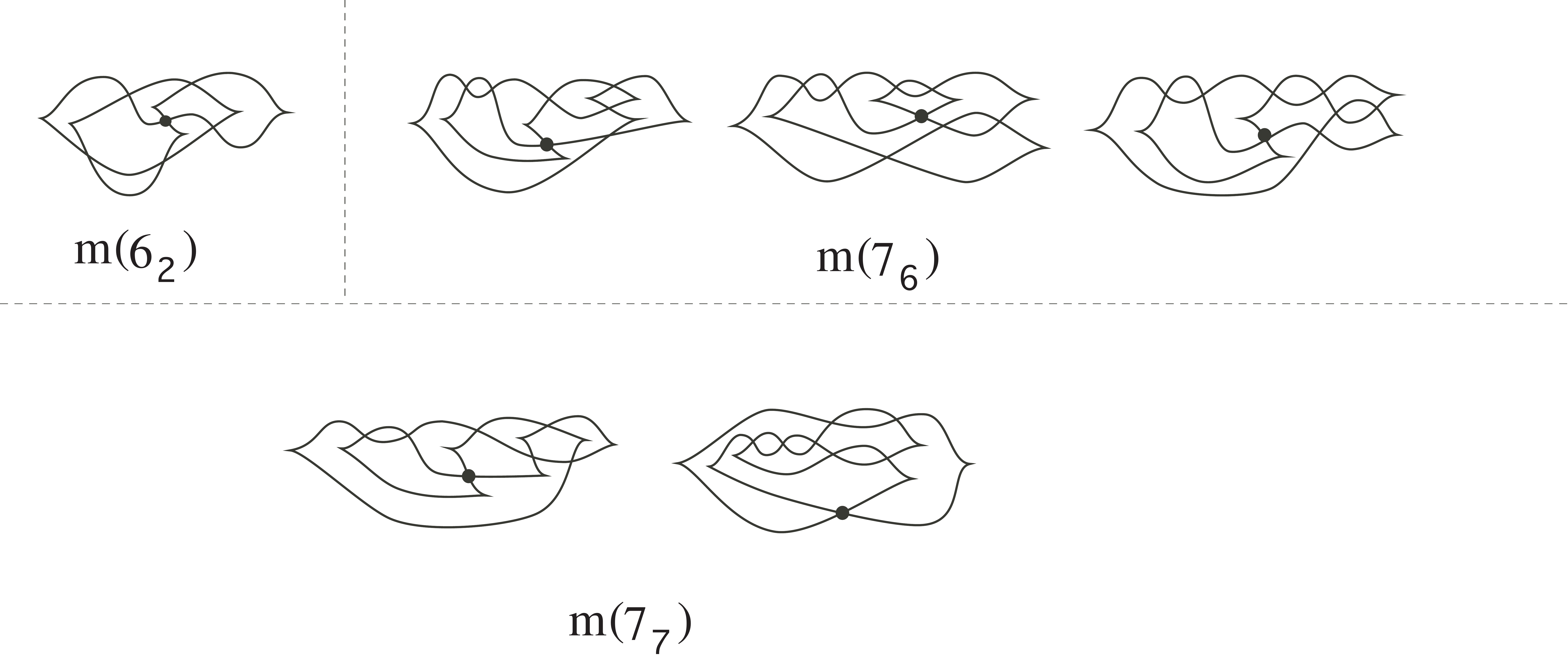}}
  \caption{Front projections representing all conjectured Legendrian representatives of $m(6_2)$, $m(7_6)$ and $m(7_7)$ with maximal Thurston-Bennequin
  invariant.  
  Each of these can be topologically unknotted by changing the indicated positive crossing in S or Hooked-X form. }
  \label{fig:pos-pt-unknot}
\end{figure}

\end{ex}

The proofs of Lemmas~\ref{lem:neg-cross-unknot} and ~\ref{lem:pos-cross-unknot} in fact show 
that if the Legendrian knot $\Lambda$ has a front projection that can
   be  topologically unknotted by changing $\nu$ negative crossings  and $\rho$  crossings
   in S, Z, or hooked-X form, then 
   $\lso(\Lambda) \leq 2\nu + 2\rho.$
  However, for our calculations we did not need this more general form.  
 
 \bibliographystyle{plain}
 \bibliography{main-L}

\begin{thebibliography}{10}

\bibitem{adams}
C.~Adams.
\newblock {\em The knot book: an elementary introduction to the mathematical
  theory of knots}.
\newblock American Mathematical Society, Providence, RI, 2004.

\bibitem{akbulut-matveyev}
S.~Akbulut and R.~Matveyev.
\newblock Exotic structures and adjunction inequality.
\newblock {\em Turkish Journal of Math}, 21:47--53, 1997.

\bibitem{bourgeois-sabloff-traynor}
F.~Bourgeois, J.~Sabloff, and L.~Traynor.
\newblock Constructions of {L}agrangian cobordisms between {L}egendrian
  submanifolds.
\newblock In Preparation.

\bibitem{casson-gordon}
A.~Casson and C.~Gordon.
\newblock Cobordism of classical knots.
\newblock In {\em A la Recherche de la Topologie Perdue}, volume~62 of {\em
  Progress in Mathematics}, pages 181--197. Birk\"auser, Boston, 1986.

\bibitem{knot-info}
J.C. Cha and C.~Livingston.
\newblock Knotinfo: Table of knot invariants, 2012.
\newblock http://www.indiana.edu/~knotinfo.

\bibitem{chantraine}
B.~Chantraine.
\newblock On {L}agrangian concordance of {L}egendrian knots.
\newblock {\em Algebr. Geom. Topol.}, 10:63--85, 2010.

\bibitem{chongchitmate-ng}
W.~Chongchitmate and L.~Ng.
\newblock An atlas of {L}egendrian knots.
\newblock {\em Exp. Math.}, 2012.
\newblock (to appear).

\bibitem{conway}
J.~Conway.
\newblock An enumeration of knots and links, and some of their algebraic
  properties.
\newblock In {\em Proc. Conf. Oxford 1967}, page 329Ð358, 1970.

\bibitem{dalton}
J.~Dalton.
\newblock {\em Legendrian Torus Links}.
\newblock PhD thesis, Bryn Mawr College, 2008.

\bibitem{eliashberg-fraser}
Y.~Eliashberg and M.~Fraser.
\newblock Topologically trivial {L}egendrian knots.
\newblock {\em J. Symplectic Geom.}, 7(2):77--127, 2009.

\bibitem{etnyre:knot-survey}
J.~Etnyre.
\newblock Legendrian and transversal knots.
\newblock In {\em Handbook of knot theory}, pages 105--185. Elsevier B. V.,
  Amsterdam, 2005.

\bibitem{etnyre-honda:knots}
J.~Etnyre and K.~Honda.
\newblock Knots and contact geometry {I}: Torus knots and the figure eight
  knot.
\newblock {\em J. Symplectic Geom.}, 1(1):63--120, 2001.

\bibitem{etnyre-ng-vertesi}
J.~Etnyre, L.~Ng, and V.~V\'ertesi.
\newblock Legendrian and transverse twist knots.
\newblock {\em J. Eur. Math. Soc.}, 2012.
\newblock (to appear).

\bibitem{kronheimer-mrowka}
P.~Kronheimer and T.~Mrowka.
\newblock Gauge theory for embedded surfaces, i.
\newblock {\em Topology}, 32(4):773--826, 1993.

\bibitem{lisca-matic}
P.~Lisca and G.~Mati{\'c}.
\newblock Stein $4$-manifolds with boundary and contact structures.
\newblock {\em Topology Appl.}, 88:55--66, 1998.

\bibitem{livingston-survey}
C.~Livingston.
\newblock A survey of classical knot concordance.
\newblock In {\em Handbook of knot theory}, pages 105--185. Elsevier B. V.,
  Amsterdam, 2005.

\bibitem{nakamura}
T.~Nakamura.
\newblock Four-genus and unknotting number of positive knots and links.
\newblock {\em Osaka J. Math}, 37:441--451, 2000.

\bibitem{rudolph}
L.~Rudolph.
\newblock An obstruction to sliceness via contact geometry and ``classical''
  gauge theory.
\newblock {\em Invent. Math}, 119:155--163, 1995.

\bibitem{schneider}
G.~Schneider.
\newblock {\em On the classification of Legendrian rational tangles via
  characteristic foliations of compressing discs}.
\newblock PhD thesis, State University of New York at Buffalo, 2011.

\bibitem{signed-unknot}
C.~Soteros, K.~Ishihara, K~Shimokawa, M.~Szafron, and M.~Vazquez.
\newblock Signed unknotting number and knot chirality discrimination via strand
  passage.
\newblock {\em Prog. Theor. Phys.}, Supplement(191):78--95, 2011.

\bibitem{traynor:strat}
L.~Traynor.
\newblock A legendrian stratification of rational tangles.
\newblock {\em Journal of Knot Theory and Its Ramifications}, 7(5):659--700,
  1998.

\end{thebibliography}

 \end{document}